\documentclass[a4paper,11pt,fleqn]{article}

\usepackage{graphicx}
\usepackage{relsize}
\usepackage{amsfonts}
\usepackage{verbatim}
\usepackage{xifthen}
\usepackage{xfrac}
\usepackage{xargs}
\usepackage{xcolor}
\usepackage[hidelinks]{hyperref}
\usepackage{physics}
\usepackage{amsmath} 
\usepackage{amsthm}  
\usepackage{amssymb} 
\usepackage{listings} 
\usepackage{tikz}
\usepackage{authblk}
\usepackage{mathtools}
\usepackage[a4paper,top=1in,bottom=1in,left=1in,right=1in]{geometry}
\usepackage{bm}
\usepackage{subcaption}
\usepackage{float}
\usepackage{algorithm}
\usepackage{algpseudocode}

\usepackage{tocloft}
\usepackage{booktabs}
\usepackage{enumitem}

\usetikzlibrary{positioning}
\usetikzlibrary{arrows}
\usetikzlibrary{decorations.pathmorphing}

\renewcommand{\bar}[1]{\mkern 1.5mu\overline{\mkern-1.5mu#1\mkern-1.5mu}\mkern 1.5mu}

\newcommand {\magn}[1]{\left| #1 \right|}
\renewcommand {\vec}[1]{\boldsymbol{#1}}

\newcommand \I{\mathrm{i}}
\newcommandx\D[1][1=]{\mathrm{d}#1}

\newcommandx {\equals}[3][1={=}]{\underbrace{#2}_{#1#3}}

\newcommandx {\CouLombV}[5][3=,5=]{\frac{#1}{\magn{\vec{#2}_{#3}-\vec{#4}_{#5}}}}
\newcommandx {\CouLombF}[4][2=,4=]{\frac{\vec{#1}_{#2}-\vec{#3}_{#4}}{\magn{\vec{#1}_{#2}-\vec{#3}_{#4}}^3}}

\newcommandx {\GFun}[4][2=,4=]{G\left(\vec{#1}_{#2},\vec{#3}_{#4}\right)}
\newcommandx {\der}[2][1=]{\frac{\D[#1] }{\D[#2]}}
\newcommandx {\pder}[2][1=]{\frac{\partial #1}{\partial #2}}
 
\renewcommand{\abs}[1]{\lvert #1 \rvert}

\renewcommand{\phi}{\varphi}

\newcommand{\R}{\mathbb{R}}
\newcommand{\C}{\mathbb{C}}
\newcommand{\Z}{\mathbb{Z}}

\theoremstyle{definition}
\newtheorem{theorem}{Theorem}[section] 
\newtheorem{lemma}[theorem]{Lemma}    

\theoremstyle{definition} 
\newtheorem{definition}[theorem]{Definition}

\theoremstyle{definition} 
\newtheorem{remark}[theorem]{Remark}

\theoremstyle{definition}
\newtheorem{proposition}[theorem]{Proposition}
\newtheorem{corollary}[theorem]{Corollary}

\renewcommand{\norm}[1]{\lVert #1\rVert}
\newcommand{\polar}{\operatorname{polar}}
\newcommand{\Ran}{\operatorname{Range}}
\newcommand{\U}{\mathrm{U}}

\begin{document}

\begin{center}
	\begin{minipage}[t]{6.0in}
Eigensolvers return orthonormal bases of invariant subspaces, but each basis is determined only up to a unitary factor, so bases
computed at different parameter values cannot, in general, be interpolated directly. Discrete parallel transport
resolves this ambiguity by aligning consecutive bases through the polar factors of their overlaps, but approximates the
continuous parallel-transported frame only to second order. We show that the resulting frames can nevertheless be
interpolated with spectral accuracy. For analytic families of invariant subspaces sampled at $N+1$ Chebyshev nodes, the
computed frames are exact samples of an analytic frame related to the parallel-transported frame by a gauge that
remains within $O(N^{-2})$ of the identity throughout a fixed complex neighborhood, uniformly in
$N$. Consequently, Chebyshev interpolation approximates this analytic frame with geometric convergence. Moreover, the
continuous parallel-transported frame can be recovered to spectral accuracy by solving a small matrix differential
equation, without additional eigensolves or derivatives of the underlying matrix family. The results do not require the family to be Hermitian and remain valid through eigenvalue crossings and branch points,
provided the selected invariant subspace depends analytically on the parameter. Numerical examples confirm the predicted rates.

		\thispagestyle{empty}

		\vspace{ -100.0in}

	\end{minipage}
\end{center}

\vspace{ 3.60in}

\begin{center}
	\begin{minipage}[t]{4.4in}
		\begin{center}

\textbf{Second-Order Steps, Spectral Accuracy: Discrete Parallel Transport of Invariant Subspaces}

			\vspace{ 0.50in}

			Hanwen Zhang$\mbox{}^{\dagger}$ \\
			\today

		\end{center}
		\vspace{ -100.0in}
	\end{minipage}
\end{center}

\vspace{ 2.00in}

\vfill
\noindent
$\mbox{}^{\dagger}$ Courant Institute of Mathematical Sciences, New York University, New York, NY 10012

\vfill
\eject

\tableofcontents

\section{Introduction}\label{sec:intro}

The computation of smoothly varying eigenvectors, or more generally of bases of invariant subspaces, is a basic problem in numerical
linear algebra and computational physics. It arises, for example, in the continuation of invariant
subspaces~\cite{DDF00,DF01,BDF08}, analytic singular value decompositions~\cite{BBMN91}, eigenvector
tracking~\cite{DPP13}, diabatic bases in quantum chemistry~\cite{LRS22}, and the construction of Wannier
functions~\cite{MV97,CLPS17,GLS19,Vanderbilt18}. Consider a family of matrices $A(k)\in\C^{m\times m}$, $k\in[-1,1]$, and suppose that a selected $n$-dimensional invariant subspace depends
analytically on $k$. An eigensolver at a parameter value $k_j$ returns an orthonormal basis $Q_j\in\C^{m\times n}$ of
the subspace, but this basis is determined only up to multiplication on the right by a unitary matrix. Consequently,
the bases returned at different parameter values need not vary smoothly, even though the subspaces do, and they cannot
be interpolated directly. The orthogonal projector $P(k)$ onto the subspace is independent of this freedom, but in
many applications $n\ll m$, and the $m\times m$ matrix $P(k)$ is too expensive to store or interpolate. The objective is therefore to construct a
smooth orthonormal frame of the subspaces from the independently computed bases.

A natural choice is the parallel-transported frame $W(k)\in\C^{m\times n}$, the solution of $W'=P'W$ from a given
initial frame (Definition~\ref{def:pt}). It satisfies $W^*W'=0$, so it does not rotate within the subspaces beyond what their motion requires,
and among all frames of the subspaces it has the smallest derivative (Proposition~\ref{prop:spread}). The same frame appears as Kato's adiabatic evolution~\cite{Kato50,Kato}, as the
parallel-transported diabatic basis in quantum chemistry~\cite{LRS22}, and as the minimum-variation basis in the
continuation of invariant subspaces~\cite{BDF08}. It is also the starting point of constructions of Wannier
functions~\cite{CHN16,GZ25}. Its differential equation, however, involves the derivative $P'$,
whereas often only the subspaces at finitely many parameter values are available, through an
eigensolver, and high-order integration costs considerably more than one eigensolve per step~\cite{DPP13}. Moreover, $P'$ involves the inverse of the gap between the eigenvalues of the subspace and
the rest of the spectrum, so its evaluation becomes ill-conditioned when an eigenvalue outside the subspace approaches those of the subspace.
Including its eigenvector in the subspace removes the difficulty, but an integrator that marches in $k$ cannot detect
it in advance.

A widely used alternative is discrete parallel transport (Section~\ref{sec:dt}), which appears in the construction of Wannier
functions~\cite{MV97,Vanderbilt18}, the continuation of invariant subspaces~\cite{BDF08}, and eigenvector
tracking~\cite{DPP13}. Given nodes $-1=k_0<k_1<\dots<k_N=1$ and a frame $V_0$
at $k_0$, it projects the current frame onto the subspace at the next node and replaces the projection $X=P(k_{j+1})V_j$ by the
nearest matrix with orthonormal columns, its polar factor $\polar(X)=X(X^*X)^{-1/2}$:
\begin{equation*}
	V_{j+1}=\polar\bigl(P(k_{j+1})V_j\bigr), \qquad j=0,\dots,N-1 .
\end{equation*}
When the eigensolver supplies a frame $Q_{j+1}$ at $k_{j+1}$, the step needs only the singular value decomposition of
the $n\times n$ overlap $Q_{j+1}^*V_j$, and $P$ is never formed (Algorithm~\ref{alg:dt}). In contrast with a marching
integrator, the eigensolves at the nodes do not depend on the transport, so the eigenvalues at all nodes
are known before the transport starts, and the subspace can be chosen to include the eigenvectors of the eigenvalues that approach it.

Discrete parallel transport approximates the parallel-transported frame with $W(-1)=V_0$ to second order: if $h$ is
the largest node spacing, then
\begin{equation}\label{eq:introsecond}
	\max_{0\le j\le N}\norm{V_j-W(k_j)}=O(h^2)
\end{equation}
(\cite{DPP13,Harland23} and Proposition~\ref{prop:second}). This difference is treated as a discretization error, to be reduced by refining the
nodes, uniformly or adaptively~\cite{CLPS17,GLS19,DF01,DPP13}, or by schemes of higher
order~\cite{BBMN91,DPP13,Harland23}. Related quantities computed from the same overlaps, such as Berry phases and
observables derived from the Berry connection, are also accurate to $O(h^2)$, and their errors have been reduced by
extrapolation~\cite{BV96,WST16} or by higher-order formulas~\cite{Roman06}.

The main observation of this paper is that the standard second-order step already produces data that can be
interpolated with spectral accuracy. At Chebyshev nodes, the natural nodes for polynomial interpolation on an
interval, the largest spacing is $h\le\pi/N$, so the computed frames differ from the parallel-transported frame by
$O(N^{-2})$, yet the Chebyshev
coefficients of their interpolant decay geometrically, far below the $O(N^{-2})$ level, down to rounding errors
(Figure~\ref{fig:ex1}(a) and~(b)). The difference~\eqref{eq:introsecond} is therefore \emph{not} an unstructured discretization error. To the
best of our knowledge, this spectral interpolation phenomenon has not been reported previously. This paper identifies
the structure of the difference, proves the spectral accuracy of the interpolation, and recovers
the parallel-transported frame itself to the same accuracy. The missing ingredient for spectral accuracy is thus not a
higher-order step, but the analytic structure of the accumulated difference.

Our main result, Theorem~\ref{thm:main}, explains this phenomenon. The computed frames are exact samples of an
analytic frame $V$ of the same subspaces, related to the parallel-transported frame $W$ by $V=WU$. The $n\times n$
matrix function $U$, which we call the gauge, is unitary on $[-1,1]$, and at the nodes
$\norm{U(k_j)-I}=\norm{V_j-W(k_j)}$ is the second-order difference in~\eqref{eq:introsecond}. Unlike classical
interpolation of a fixed analytic function, here $V$ changes with $N$: refining the nodes changes the discrete
transport, and hence the analytic frame $V$ sampled by the computed frames. Nevertheless, the theorem shows that for
all sufficiently large $N$, $U-I$ remains $O(N^{-2})$ throughout a fixed complex neighborhood of the interval, not merely at the nodes,
so $V$ stays within $O(N^{-2})$ of the fixed frame $W$ there. In particular, these frames are uniformly bounded on a
common region of analyticity, so Chebyshev interpolation of the computed frames approximates $V$ with geometrically
decaying error, as for a fixed analytic function.

This distinction matters in computation. The interpolant $\mathcal I_NV$ approximates $V$ to spectral accuracy, but in general
$\mathcal I_NV-W=O(N^{-2})$, due to the gauge. When $W$ itself is needed, the gauge can be removed. Although $U$ is
not available from the computed data, since $U_j=W(k_j)^*V_j$ involves the unknown frame $W$, the connection matrix
$V^*V'=U^*U'$ of $V$ (Definition~\ref{def:conn}) can be approximated spectrally from $\mathcal I_NV$ and its derivative. Solving an $n\times n$
linear differential equation with this approximate connection then recovers $W$ to spectral accuracy, using the same
eigensolves and without derivatives of $P$ (Corollary~\ref{cor:recover}).

The results extend in two directions. The proofs rely on operator-norm estimates for the projectors $P(k)$ and
finite-dimensional overlap matrices, so they carry over to analytic families of finite-rank orthogonal projectors on a
Hilbert space, such as spectral projectors of differential operators. Since the constants do not depend explicitly on
the ambient dimension, the bounds also hold uniformly under discretization, provided the projector families have uniform
analytic bounds. For (quasi-)periodic families, equispaced nodes and trigonometric interpolation take the place of
Chebyshev nodes and Chebyshev interpolation. We study this setting in a forthcoming companion paper.

The paper is organized as follows. Section~\ref{sec:notation} contains the notation and definitions used in this paper. Section~\ref{sec:setting} introduces
the parallel-transported frame and discrete parallel transport, and explains how the discrete transport is computed
from the output of an eigensolver. Section~\ref{sec:main} states the main results, which are proved in
Section~\ref{sec:proofs} using the estimates collected in Section~\ref{sec:estimates}. Section~\ref{sec:numerics}
presents numerical examples, and Section~\ref{sec:conclusion} discusses generalizations and applications.

\section{Notation and definitions}\label{sec:notation}

This section collects the notation and definitions used throughout the paper.

For $X\in\C^{m\times n}$, $X^*\in\C^{n\times m}$ is the conjugate transpose. We write $I$ (or $I_n$) for the
$n\times n$ identity matrix and $\U(n)\subset\C^{n\times n}$ for the group of unitary matrices. Vectors $x\in\C^n$ carry
the Euclidean norm $\norm x$. For $X\in\C^{m\times n}$ we use the spectral norm and the Frobenius norm, and for a
function $F\colon K\to\C^{m\times n}$ on a set $K$ the supremum of the spectral norm:
\begin{equation}\label{eq:norms}
	\norm{X}=\max_{x\in\C^n,\ \norm x=1}\norm{Xx}, \qquad \norm{X}_{\mathrm F}=(\operatorname{tr}X^*X)^{1/2}, \qquad
	\norm{F}_K=\sup_{k\in K}\norm{F(k)} .
\end{equation}

\paragraph{Polar factor.} For $X\in\C^{m\times n}$ of rank $n$,
\begin{equation}\label{eq:polar}
	\polar(X):=X\,(X^*X)^{-1/2}\in\C^{m\times n}
\end{equation}
is the factor with orthonormal columns in the polar decomposition $X=\polar(X)\,(X^*X)^{1/2}$ (Theorem~8.1
in~\cite{Higham08}). Here $(X^*X)^{1/2}$ is the unique Hermitian positive definite square root of $X^*X$. Since $(YX)^*(YX)=X^*X$ and $(XY)^*(XY)=Y^*(X^*X)Y$, the definition~\eqref{eq:polar} gives:
\begin{enumerate}[label=(P)]
	\item\label{P} If $Y$ has orthonormal columns, then $\polar(YX)=Y\polar(X)$. If $Y\in\U(n)$, then
	$\polar(XY)=\polar(X)\,Y$.
\end{enumerate}

\paragraph{Chebyshev polynomials and Bernstein ellipses.} The Chebyshev polynomial $T_l$ of degree $l$ satisfies
$T_l(\cos\theta)=\cos l\theta$ for $\theta\in\C$. The Chebyshev points of the second kind are the extreme points of
$T_N$ on $[-1,1]$. In increasing order, they are
\begin{equation}\label{eq:chebnodes}
	k_j=-\cos\frac{\pi j}N, \qquad j=0,\dots,N,
\end{equation}
and by the mean value theorem their largest spacing is at most $\pi/N$. For a function $F$ with values in
$\C^{p\times q}$ that is defined at the nodes, its Chebyshev interpolant $\mathcal I_NF$ is the polynomial in $k$ of
degree at most $N$, with coefficients in $\C^{p\times q}$, such that
\begin{equation}\label{eq:interp}
	(\mathcal I_NF)(k_j)=F(k_j), \qquad j=0,\dots,N.
\end{equation}
For $\rho>1$, $\mathcal E_\rho$ is the open region bounded by the Bernstein ellipse
\begin{equation}\label{eq:ellipse}
	\bigl\{\tfrac12(z+z^{-1}):\ \abs z=\rho\bigr\},
\end{equation}
which contains $[-1,1]$. For $\beta>0$, $\mathcal S_\beta=\{\theta\in\C:\abs{\operatorname{Im}\theta}<\beta\}$ is the
open strip of half-width $\beta$. As $\theta$ runs over $\mathcal S_\beta$, $z=e^{\I\theta}$ runs over the annulus
$e^{-\beta}<\abs z<e^\beta$, which $z\mapsto\frac12(z+z^{-1})$ maps onto $\mathcal E_{e^\beta}$, so
\begin{equation}\label{eq:striptoellipse}
	\cos(\mathcal S_\beta)=\mathcal E_{e^\beta} .
\end{equation}

\paragraph{Analytic matrix families.} A function $X\colon\mathcal D\to\C^{m\times n}$ on an open set
$\mathcal D\subset\C$ is analytic if each of its entries is. A function on an open interval $J\subset\R$ is analytic if
it is the restriction to $J$ of an analytic function on an open set of $\C$ containing $J$. If $\mathcal D$ is
symmetric about the real axis, the reflection of an analytic $X$,
\begin{equation}\label{eq:sharp}
	X^\sharp(z):=X(\bar z)^*, \qquad z\in\mathcal D,
\end{equation}
is analytic, equals $X^*$ on $\mathcal D\cap\R$, and satisfies
\begin{equation}\label{eq:sharpprop}
	X^{\sharp\sharp}=X, \qquad (XY)^\sharp=Y^\sharp X^\sharp .
\end{equation}
It replaces the adjoint, which is not analytic, when functions of $X$ and $X^*$ are continued off the real axis.

\section{The parallel-transported frame and discrete parallel transport}\label{sec:setting}

This section defines the parallel-transported frame of an analytic family of subspaces (Section~\ref{sec:pt}) and
discrete parallel transport, which approximates it from the subspaces at a set of nodes (Section~\ref{sec:dt}).
Section~\ref{sec:compute} describes how the step is computed.

\subsection{Analytic families of subspaces and parallel transport}\label{sec:pt}

Suppose that a family of $n$-dimensional subspaces of $\C^m$ is given by the orthogonal projectors onto them.
Precisely, let $J\subset\R$ be an open interval, and let $P\colon J\to\C^{m\times m}$ be analytic, with
$P(k)$ an orthogonal projector of rank $n$ for every $k$:
\begin{equation}\label{eq:proj}
	P(k)^2=P(k)=P(k)^*, \qquad \dim\Ran P(k)=n, \qquad k\in J.
\end{equation}
The subspaces are then $\Ran P(k)$, $k\in J$.

The most important example in practice is the Hermitian case. Let $H\colon J\to\C^{m\times m}$ be analytic, with
$H(k)$ Hermitian for every $k\in J$. By Rellich's theorem (Section~II.6 in~\cite{Kato}), the eigenvalues and an
orthonormal basis of eigenvectors of $H$ can be followed analytically along $J$, including through crossings of
eigenvalues: there are analytic functions $\lambda_1,\dots,\lambda_m\colon J\to\R$ and
$u_1,\dots,u_m\colon J\to\C^m$ such that $H(k)u_i(k)=\lambda_i(k)u_i(k)$ and $u_1(k),\dots,u_m(k)$ are orthonormal
for every $k\in J$. Hence, for any $n$ distinct indices $i_1,\dots,i_n$, the orthogonal projector onto the span of the corresponding
eigenvectors,
\begin{equation}\label{eq:Pherm}
	P(k)=\sum_{r=1}^n u_{i_r}(k)\,u_{i_r}(k)^* , \qquad k\in J,
\end{equation}
is analytic and satisfies~\eqref{eq:proj}. This includes the spectral projector of a group of $n$ eigenvalues
separated from the rest of the spectrum, and the projector onto eigenvectors whose eigenvalues cross other
eigenvalues.

The setting also covers invariant subspaces of a non-normal family $A\colon J\to\C^{m\times m}$. For a
group of $n$ eigenvalues that stays separated from the rest of the spectrum, the Riesz projector onto the corresponding
invariant subspace is analytic but in general not orthogonal. We take $P(k)$ to be the orthogonal projector onto
the same subspace. This projector is also analytic, since the invariant subspace has an analytic basis
$X\colon J\to\C^{m\times n}$ (Section~II.4.2 in~\cite{Kato}) and $P=X(X^*X)^{-1}X^*$. In both cases, the family
$H$ or $A$ enters the rest of the paper only through $P$.

\begin{definition}[Frames]\label{def:frame}
A \emph{frame at $k\in J$} is a matrix $Q\in\C^{m\times n}$ with $Q^*Q=I_n$ and $\Ran Q=\Ran P(k)$. A \emph{frame
of $P$} is a differentiable function $F\colon J\to\C^{m\times n}$ with $F(k)$ a frame at $k$ for every $k\in J$.
\end{definition}

If $Q$ is a frame at $k$, then $P(k)=QQ^*$, and the frames at $k$ are the matrices $QU$ with $U\in\U(n)$. For $n=1$,
$U$ is the familiar phase $e^{\I\phi}$ of a normalized eigenvector.

\begin{definition}[Parallel-transported frame]\label{def:pt}
Let $k_0\in J$, and let $V_0$ be a frame at $k_0$. The \emph{parallel-transported frame} started from $V_0$ is the
solution $W\colon J\to\C^{m\times n}$ of
\begin{equation}\label{eq:kato}
	W'(k)=P'(k)\,W(k), \quad k\in J, \qquad W(k_0)=V_0,
\end{equation}
where the prime denotes $\D/\D k$.
\end{definition}

For a non-normal family, $W$ is transported with the orthogonal projector, not with the Riesz projector, which would
not keep the frames orthonormal.

As a linear differential equation with an analytic coefficient, \eqref{eq:kato} has a unique solution on all of $J$,
and it is analytic. For every $k\in J$, $W(k)$ is a frame at $k$, and
\begin{equation}\label{eq:katoprop}
	W^*W'=0 ;
\end{equation}
see Remark~II.4.1 and Section~II.6.2 in~\cite{Kato}. Both facts follow from $PP'P=0$, which is the derivative of $P^2=P$ multiplied by $P$ on both sides, together with
the uniqueness of solutions of linear differential equations.
The frame $W$ is called parallel-transported since, by~\eqref{eq:katoprop}, $PW'=WW^*W'=0$: its derivative has no
component in the subspace.

The following lemma shows that the parallel-transported frame continues off the real axis together with $P$. The
continuation follows from Theorem~4.7 in~\cite{Teschl12} on linear differential equations with analytic coefficients. On
$J$, the identities~\eqref{eq:katoext} read $W^*W=I$ and~\eqref{eq:katoprop}, and the identity theorem (Theorem~4.8 in
Chapter~2 of~\cite{SteinShakarchi03}) extends them off the real axis.

\begin{lemma}[Analytic continuation of the parallel-transported frame]\label{lem:continue}
Let $\mathcal D\subset\C$ be a simply connected open set that contains $J$, and suppose that $P$ extends to an analytic
function on $\mathcal D$. Then so does $W$, and $W'=P'W$ on $\mathcal D$. If $\mathcal D$ is also symmetric about the
real axis, then
\begin{equation}\label{eq:katoext}
	W^\sharp W=I, \qquad W^\sharp W'=0 \qquad\text{on }\mathcal D.
\end{equation}
\end{lemma}

\begin{remark}[Reparametrization]\label{rmk:param}
The parallel-transported frame does not depend on the parametrization: if $\tilde J\subset\R$ is an open interval,
$\varphi\colon\tilde J\to J$ is analytic and $\varphi(s_0)=k_0$, then
$(W\circ\varphi)'=\varphi'\,(P'W)\circ\varphi=(P\circ\varphi)'\,(W\circ\varphi)$, so $W\circ\varphi$ is the
parallel-transported frame of $P\circ\varphi$ started from $V_0$ at $s_0$.
\end{remark}

\begin{definition}[Connection matrix]\label{def:conn}
The \emph{connection matrix} of a frame $F$ of $P$ is $F^*F'\colon J\to\C^{n\times n}$.
\end{definition}

The connection matrix is skew-Hermitian, since differentiating $F^*F=I_n$ gives $F^*F'+(F^*F')^*=0$. Every frame of $P$ is $F=WU$ with
$U:=W^*F\colon J\to\U(n)$, since $F=PF=WW^*F$. We call $U$ the \emph{gauge} of $F$ relative to $W$. The term comes from
physics, where a change of frame is a gauge transformation~\cite{Frankel11,MMYSV12}. By~\eqref{eq:katoprop}, the connection
matrix of $F$ is
\begin{equation}\label{eq:conn}
	F^*F'=U^*W^*(W'U+WU')=U^*U' .
\end{equation}
Hence a frame of $P$ has connection matrix zero exactly when it solves the parallel-transport equation~\eqref{eq:kato}:
since $F'=W'U+WU'=P'F+WU'$,
\begin{equation}\label{eq:zeroconn}
	F^*F'=0 \quad\Longleftrightarrow\quad F'=P'F .
\end{equation}
Hence the parallel-transported frame is the frame of $P$ with connection matrix zero that starts from $V_0$, and the
other frames of $P$ differ from it by their gauge. The matrix $F^*F'$ is referred to as the connection matrix since it represents in the
frame $F$ the connection that $\C^m$ induces on the family of subspaces~\cite{Harland23}. In physics, the Hermitian
matrix $\I F^*F'$ is called the Berry connection (see~\cite{MV97} and Sections~3.1.2 and~3.6.4
in~\cite{Vanderbilt18}).

From a computational point of view, the parallel-transported frame has the advantage that it varies as slowly as
the subspace allows: among the frames of $P$, it has the smallest derivative and the smallest second moment of
its Chebyshev coefficients, as the following proposition shows. The splitting~\eqref{eq:pythag} is in the proof of Lemma~2.1
in~\cite{BDF08}, and the second moment in~\eqref{eq:moment} is the analogue of the term $\sum_n\langle r^2\rangle_n$ in the spread
of Wannier functions (Section~III in~\cite{MV97}).

\begin{proposition}[Concentration of Chebyshev coefficients]\label{prop:spread}
Suppose that $J$ contains $[-1,1]$, and let $F$ be an analytic frame of $P$, with the Chebyshev series
$F(k)=\sum_{l\ge0}C_lT_l(k)$ on $[-1,1]$. Then
\begin{gather}
	\norm{F'(k)}_{\mathrm F}^2=\norm{W'(k)}_{\mathrm F}^2+\norm{F(k)^*F'(k)}_{\mathrm F}^2, \qquad k\in J, \label{eq:pythag}\\
	\frac\pi2\sum_{l\ge1}l^2\norm{C_l}_{\mathrm F}^2=\int_{-1}^1\norm{F'(k)}_{\mathrm F}^2\sqrt{1-k^2}\,\D k . \label{eq:moment}
\end{gather}
In particular, among the analytic frames of $P$, the second moment $\sum_{l\ge1}l^2\norm{C_l}_{\mathrm F}^2$ is smallest
exactly for the parallel-transported frames.
\end{proposition}

\begin{proof}
We first prove~\eqref{eq:pythag}. Let $U$ be the gauge of $F$, so $F=WU$. The two terms of $F'=P'F+WU'$ are
orthogonal, since $PP'F=PP'PF=0$ and $WU'=PWU'$. Since $\norm{P'F}_{\mathrm F}=\norm{W'U}_{\mathrm F}=\norm{W'}_{\mathrm F}$ and
$\norm{WU'}_{\mathrm F}=\norm{U'}_{\mathrm F}=\norm{U^*U'}_{\mathrm F}=\norm{F^*F'}_{\mathrm F}$ by~\eqref{eq:conn}, this gives~\eqref{eq:pythag}. Next,
\eqref{eq:moment} is Parseval's identity for the derivative of $F(\cos\theta)=\sum_{l\ge0}C_l\cos l\theta$ on
$[0,\pi]$, written in the variable $k=\cos\theta$. The last statement follows from~\eqref{eq:pythag}
and~\eqref{eq:moment}: equality forces $F^*F'=0$ on $[-1,1]$, hence on $J$ by the identity theorem,
and~\eqref{eq:zeroconn} shows that $F^*F'=0$ exactly for the parallel-transported frames.
\end{proof}

\subsection{Discrete parallel transport}\label{sec:dt}

The discrete parallel transport approximates the parallel-transported frame using
only the values of $P$ at a set of nodes and not its derivative. Each step projects the current frame onto the subspace
at the next node and restores orthonormal columns.

\begin{definition}[Discrete parallel transport]\label{def:dt}
Let $k_0,k_1,\dots,k_N\in J$ be nodes, with $k_0$ and $V_0$ as in Definition~\ref{def:pt}. The \emph{discrete
parallel transport} defines matrices $V_1,\dots,V_N$ by
\begin{equation}\label{eq:dt}
	V_{j+1}:=\polar\bigl(P(k_{j+1})\,V_j\bigr), \qquad j=0,\dots,N-1,
\end{equation}
as long as $P(k_{j+1})V_j$ has rank $n$.
\end{definition}

The computed frames $V_j$ approximate the parallel-transported frame at the nodes,
\begin{equation}\label{eq:Wj}
	W_j=W(k_j), \qquad j=0,\dots,N.
\end{equation}
The function $W$ itself is not computed.

When a frame at the next node is available, the step can be computed from an $n\times n$ overlap. If $Q_{j+1}$ is a
frame at $k_{j+1}$, then $P(k_{j+1})V_j=Q_{j+1}(Q_{j+1}^*V_j)$, so the step is defined exactly when the $n\times n$
overlap $Q_{j+1}^*V_j$ is invertible. In that case,~\ref{P} gives
\begin{equation}\label{eq:svd}
	V_{j+1}=Q_{j+1}\polar(Q_{j+1}^*V_j)=Q_{j+1}XY^*, \qquad Q_{j+1}^*V_j=X\Sigma Y^* \ \text{(SVD)},
\end{equation}
where the second equality is~\eqref{eq:polar} with $(X\Sigma Y^*)^*X\Sigma Y^*=Y\Sigma^2Y^*$. The result does not depend
on the choice of $Q_{j+1}$, since $V_{j+1}$ is determined by $P(k_{j+1})$ and $V_j$.

Both forms are widely used. For Wannier functions, the projector form~\eqref{eq:dt} is used
in~\cite{CLPS17,GLS19}, and the SVD form~\eqref{eq:svd} is the parallel transport of Marzari and Vanderbilt
(Section~IV.C.1 in~\cite{MV97} and Section~3.6.2 in~\cite{Vanderbilt18}). For analytic singular value decompositions, the SVD form is a Procrustes step~\cite{BBMN91}. In the continuation of invariant subspaces, the step gives the orthonormal
basis at the next parameter value nearest to the current one, computed from a Riccati equation in~\cite{DDF00,DF01}
and by an SVD in~\cite{BDF08}. Applied column by column, it is the corrector used to track eigenvectors of
Hermitian families in~\cite{DPP13}.

\begin{remark}[A fixed reference]\label{rmk:fixedref}
Replacing $V_j$ by $V_0$ in~\eqref{eq:dt} aligns every basis with one fixed frame,
$\polar\bigl(P(k_j)V_0\bigr)=P(k_j)V_0\bigl(V_0^*P(k_j)V_0\bigr)^{-1/2}$, a construction used
in~\cite{Rheinboldt88,MR93,AF11,LRS22}. It is defined only while $V_0^*P(k)V_0$ is invertible, that is, while no
nonzero vector of $\Ran V_0$ is orthogonal to $\Ran P(k)$. Along a longer path the subspace can move far from
$\Ran V_0$, and beyond such a point the aligned bases no longer vary continuously (Example~1 in~\cite{BBMN91}).
The discrete parallel transport~\eqref{eq:dt} aligns with the current frame $V_j$ instead, so it needs only that
consecutive subspaces are close (Lemma~\ref{lem:dt}).
\end{remark}

The following lemma shows that the step~\eqref{eq:dt} is defined as long as consecutive projectors differ by less
than $1$ in norm, and that it stays in the exact subspace. The step is the unitary of Sz.-Nagy and Kato that maps
$\Ran P(k_j)$ onto $\Ran P(k_{j+1})$, applied to $V_j$ (Section~I.4.6 and Theorem~I.6.32 in~\cite{Kato}).
Lemmas~\ref{lem:dt} and~\ref{lem:rev} are classical facts about this unitary, stated for frames.

\begin{lemma}\label{lem:dt}
Suppose that $\norm{P(k_{j+1})-P(k_j)}<1$ for $0\le j<N$. Then:
\begin{itemize}
	\item[(i)] for $0\le j<N$, the matrix $G_j:=V_j^*P(k_{j+1})V_j$ is positive definite, so the step~\eqref{eq:dt} is
	defined, and
	\begin{equation}\label{eq:dtS}
		V_{j+1}=P(k_{j+1})\,V_j\,G_j^{-1/2} ;
	\end{equation}
	\item[(ii)] every $V_j$, $0\le j\le N$, is a frame at $k_j$.
\end{itemize}
\end{lemma}

\begin{proof}
We argue by induction on $j$. Assume that $V_j$ is a frame at $k_j$, which holds for $j=0$. Since $P(k_{j+1})$ is an
orthogonal projector, $G_j=\bigl(P(k_{j+1})V_j\bigr)^*P(k_{j+1})V_j$. Since $P(k_j)V_j=V_j$ and $\norm{V_jx}=\norm x$,
for $x\in\C^n$ with $x\ne0$,
\begin{equation}
	\norm{P(k_{j+1})V_jx}=\bigl\|V_jx-\bigl(P(k_j)-P(k_{j+1})\bigr)V_jx\bigr\|
	\ \ge\ \bigl(1-\norm{P(k_{j+1})-P(k_j)}\bigr)\norm x\ >\ 0 .
\end{equation}
Hence $P(k_{j+1})V_j$ has rank $n$, $G_j$ is positive definite, and the definition~\eqref{eq:polar} of the polar factor
gives~\eqref{eq:dtS}.

It remains to show that $V_{j+1}$ is a frame at $k_{j+1}$. By~\eqref{eq:dtS},
$V_{j+1}^*V_{j+1}=G_j^{-1/2}G_jG_j^{-1/2}=I_n$. Since $G_j^{-1/2}$ is invertible, \eqref{eq:dtS} also gives
$\Ran V_{j+1}=\Ran P(k_{j+1})V_j\subset\Ran P(k_{j+1})$, and both spaces have dimension $n$, so they are equal.
\end{proof}

The next lemma shows that the step can be undone exactly; the proof of Theorem~\ref{thm:main} uses this to extend the
computed frames evenly around a circle (Lemma~\ref{lem:retrace}).

\begin{lemma}[Reversibility]\label{lem:rev}
Under the assumptions of Lemma~\ref{lem:dt}, for $0\le j<N$,
\begin{equation}\label{eq:rev}
	V_j=\polar\bigl(P(k_j)\,V_{j+1}\bigr).
\end{equation}
\end{lemma}

\begin{proof}
By Lemma~\ref{lem:dt}(ii), $V_j$ is a frame at $k_j$, so $P(k_j)=V_jV_j^*$, and by~\eqref{eq:dtS}
\begin{equation}
	P(k_j)V_{j+1}=V_j\,\bigl(V_j^*P(k_{j+1})V_j\bigr)\,G_j^{-1/2}=V_j\,G_j\,G_j^{-1/2}=V_j\,G_j^{1/2}.
\end{equation}
By~\ref{P}, $\polar(V_jG_j^{1/2})=V_j\polar(G_j^{1/2})$, and $\polar(G_j^{1/2})=G_j^{1/2}G_j^{-1/2}=I_n$
by~\eqref{eq:polar}, since $G_j^{1/2}$ is Hermitian and positive definite. Hence $\polar\bigl(P(k_j)V_{j+1}\bigr)=V_j$, which is~\eqref{eq:rev}.
\end{proof}

The following proposition shows that the error $V_j-W_j$ is $O(h^2)$ on any nodes, where $h$ is the largest
spacing. We state it on $[-1,1]$, the interval of the main
results; an affine change of variable gives any compact interval (Remark~\ref{rmk:param}). For $n=1$ and $P(k)$ the
projector onto an
eigenvector of a Hermitian family with distinct eigenvalues, the bound at $k_N$ is Theorem~2.7 in~\cite{DPP13}. For
general $n$, the proposition is implicit in Harland's analysis of a closely related second-order scheme
(equation~(2.23) in~\cite{Harland23}). We give a self-contained proof in Section~\ref{sec:onestep}, since its one-step
factor is also the starting point of the proof of Theorem~\ref{thm:main}.

\begin{proposition}[Second-order convergence]\label{prop:second}
Let $P\colon J\to\C^{m\times m}$ be analytic on an open interval $J$ that contains $[-1,1]$, with $P(k)$ an orthogonal
projector of rank $n$ for every $k\in J$. Let $V_0$ be a frame at $k=-1$, and let $W\colon J\to\C^{m\times n}$ be the
parallel-transported frame started from $V_0$. There are $C$ and $h_0>0$, depending only on $P$, such that for all
nodes $-1=k_0<k_1<\dots<k_N=1$ with $h:=\max_{0\le j<N}(k_{j+1}-k_j)\le h_0$, the discrete parallel
transport started from $V_0$ is defined, and
\begin{equation}\label{eq:secondorder}
	\max_{0\le j\le N}\norm{V_j-W_j}\le Ch^2.
\end{equation}
\end{proposition}

The bound~\eqref{eq:secondorder} measures only an error in the basis. Since $V_j$ and $W_j$ are both frames at $k_j$
(Lemma~\ref{lem:dt}(ii)), the subspace $\Ran P(k_j)$ is computed exactly, and $V_j$ is $W_j$ rotated within it.

Geometrically, since $\norm{P(k_{j+1})-P(k_j)}<1$, each step is exact parallel transport along the unique minimizing
Grassmannian geodesic between $\Ran P(k_j)$ and $\Ran P(k_{j+1})$: $V_{j+1}$ is the endpoint of the
parallel-transported frame along that geodesic started from $V_j$ (Theorem~5.4 in~\cite{BZA24} for real matrices; the complex case is
analogous, Section~1 in~\cite{BZA24}). In particular, if the subspaces move along such a geodesic, the discrete
parallel transport is exact: $V_j=W_j$ at every node, for every $N$.

\subsection{Computing the step}\label{sec:compute}

The projector form~\eqref{eq:dt} and the SVD form~\eqref{eq:svd} are equivalent mathematically. Since
$P(k_{j+1})=Q_{j+1}Q_{j+1}^*$, both need only the $n\times n$ overlap $Q_{j+1}^*V_j$ and products with $Q_{j+1}$, so
each step costs $O(mn^2+n^3)$ operations, and the $m\times m$ projector is never formed. In floating-point arithmetic,
however, the projector form computes $G_j=(Q_{j+1}^*V_j)^*(Q_{j+1}^*V_j)$, whose condition number is
$\kappa(Q_{j+1}^*V_j)^2$, where $\kappa(A)=\norm A\norm{A^{-1}}$. The SVD form computes the polar factor of the overlap
directly, without forming $G_j$, and as accurately as can be expected in floating-point arithmetic (Section~8.8
in~\cite{Higham08}). Since the overlap is only $n\times n$, computing its polar factor through the SVD is the method of
choice (p.~196 in~\cite{Higham08}).

The condition number of the overlap is set by the distance between consecutive subspaces. Since the singular values
of $Q_{j+1}^*V_j$ are the cosines of the principal angles between $\Ran P(k_j)$ and $\Ran P(k_{j+1})$, and
$\norm{P(k_{j+1})-P(k_j)}$ is the sine of the largest one (Theorem~1 in~\cite{BG73} and Section~4.3
in~\cite{EAS98}),
\begin{equation}\label{eq:kappa}
	\kappa(Q_{j+1}^*V_j)\le\frac1{\sigma_j}=\bigl(1-\norm{P(k_{j+1})-P(k_j)}^2\bigr)^{-1/2}, \qquad
	\sigma_j:=\sigma_{\min}(Q_{j+1}^*V_j).
\end{equation}
Since $\norm{P(k_{j+1})-P(k_j)}=O(h)$, the bound is $1+O(h^2)$. The numbers $\sigma_j$ also check the step size at no extra cost: the step is defined exactly when $\sigma_j>0$, which
by~\eqref{eq:kappa} is the hypothesis $\norm{P(k_{j+1})-P(k_j)}<1$ of Lemma~\ref{lem:dt}.

For these reasons, the numerical examples use the SVD form, as in Algorithm~\ref{alg:dt}. The proofs use the projector
form~\eqref{eq:dt}, and the SVD form only with the parallel-transported frame $Q_{j+1}=W_{j+1}$
(Lemma~\ref{lem:factor}).

\begin{algorithm}[H]
\caption{Discrete parallel transport in the SVD form}\label{alg:dt}
\begin{algorithmic}[1]
\Require nodes $k_0,\dots,k_N\in J$; a frame $Q_j$ at $k_j$ for $j=0,\dots,N$, in any gauge; a frame $V_0$ at $k_0$;
a tolerance $\sigma_{\rm tol}\in(0,1)$
\Ensure the frames $V_1,\dots,V_N$ of~\eqref{eq:dt} and the numbers $\sigma_0,\dots,\sigma_{N-1}$
\For{$j=0,\dots,N-1$}
	\State $Q_{j+1}^*V_j=X\Sigma Y^*$ \Comment{overlap, $O(mn^2)$ operations, and its SVD, $O(n^3)$}
	\State $\sigma_j\gets\sigma_{\min}(\Sigma)$; stop if $\sigma_j<\sigma_{\rm tol}$ \Comment{otherwise
	$\kappa(Q_{j+1}^*V_j)\le1/\sigma_{\rm tol}$ by~\eqref{eq:kappa}}
	\State $V_{j+1}\gets Q_{j+1}XY^*$ \Comment{$O(mn^2)$ operations}
\EndFor
\end{algorithmic}
\end{algorithm}

\section{Main results}\label{sec:main}

Let $P$ be an analytic family of orthogonal projectors of rank $n$ as in~\eqref{eq:proj}, let $V_0$ be a frame at
$k_0=-1$, and let $W$ be the parallel-transported frame started from $V_0$ (Definition~\ref{def:pt}). The discrete
parallel transport (Definition~\ref{def:dt}) computes frames $V_j$ at the nodes $-1=k_0<k_1<\dots<k_N=1$. Each step
projects the current frame onto the next subspace and takes the polar factor, and it is defined as long as
consecutive projectors differ by less than $1$ in norm (Lemma~\ref{lem:dt}). The computed frames approximate
$W_j=W(k_j)$ with error $O(h^2)$, where $h$ is the largest spacing (Proposition~\ref{prop:second}). On the Chebyshev
nodes~\eqref{eq:chebnodes}, $h\le\pi/N$, so the error is $O(N^{-2})$.

The difference $V_j-W_j$ in~\eqref{eq:secondorder} is usually treated as a discretization error
(Section~\ref{sec:intro}). The central observation of this paper is that, on Chebyshev nodes,
the computed frames are \emph{exact} samples of another frame of $P$, which is analytic on a fixed complex
neighborhood of $[-1,1]$. Since $V_j$ and $W_j$ are frames of the same subspace (Lemma~\ref{lem:dt}), they differ by a
unitary factor,
\begin{equation}\label{eq:gaugej}
	V_j=W_jU_j, \qquad U_j:=W_j^*V_j\in\U(n), \qquad j=0,\dots,N,
\end{equation}
and the new frame is $V=WU$, where $U$ is an analytic function with $U(k_j)=U_j$ that is unitary on $[-1,1]$ and
within $O(N^{-2})$ of $I$. Then $V$ can be interpolated from the samples $V_j$ with spectral accuracy (Theorem~\ref{thm:main}(c)), and $W$
can be recovered from $V$ to the same accuracy by removing the gauge $U$, which solves an $n\times n$ linear
differential equation whose coefficient is the connection matrix of $V$ (Definition~\ref{def:conn}; Corollaries~\ref{cor:modified}
and~\ref{cor:recover}).

Theorem~\ref{thm:main} resembles the classical bound for the Chebyshev interpolation of an analytic function, but
the resemblance is deceptive. For a fixed function analytic and bounded on the region $\mathcal E_\rho$ enclosed by
the Bernstein ellipse~\eqref{eq:ellipse}, Chebyshev interpolation converges at the rate $\rho^{-N}$, with a constant proportional to
the bound (Chapter~8 in~\cite{ATAP}). Unlike in the classical setting, the interpolated function is not fixed as $N$
grows: each $N$ gives a new discrete transport, so the computed frames, and with them $V$, change with $N$. The
classical bound therefore gives the rate $\rho'^{-N}$ only if $V$ is bounded on $\mathcal E_{\rho'}$ independently of
$N$; this uniformity is the substance of the theorem. For a fixed $N$, the interpolant of the computed frames is itself analytic, so
analyticity alone gives no information about convergence. Theorem~\ref{thm:main} establishes this uniform bound, so the
computed frames are interpolated at the same rate as $P$ itself.

\begin{theorem}[Analytic gauge and spectral accuracy]\label{thm:main}
Let $P\colon J\to\C^{m\times m}$ be analytic on an open interval $J$ that contains $[-1,1]$, with $P(k)$ an orthogonal
projector of rank $n$ for every $k\in J$, and suppose that $P$ extends to an analytic function
$P\colon\mathcal E_\rho\to\C^{m\times m}$ for some $\rho>1$. Let $V_0$ be a frame at $k=-1$, let $W$ be the
parallel-transported frame started from $V_0$, and let $k_0,\dots,k_N$ be the Chebyshev nodes~\eqref{eq:chebnodes}. For every $1<\rho'<\rho$ there are $C$ and $N_0$ such that for every
$N\ge N_0$ the discrete parallel transport $V_1,\dots,V_N$ started from $V_0$ is defined, and the following hold.
\begin{itemize}
	\item[(a)] \textbf{The gauges are samples of an analytic function, unitary on $[-1,1]$.} There is an analytic function
	$U\colon\mathcal E_{\rho'}\to\C^{n\times n}$, defined on $[-1,1]$ as the polar factor of the Chebyshev
	interpolant of the gauges~\eqref{eq:gaugej}, such that
	\begin{equation}\label{eq:main2}
	\begin{gathered}
		U(k_j)=U_j, \qquad j=0,\dots,N,\\
		\norm{U(k)-I}\le\frac C{N^2}, \qquad k\in\mathcal E_{\rho'}.
	\end{gathered}
	\end{equation}
	\item[(b)] \textbf{The computed frames are samples of an analytic frame.} The function $V:=WU$ extends from $[-1,1]$
	to an analytic function on $\mathcal E_{\rho'}$, and
	\begin{equation}\label{eq:main1}
	\begin{gathered}
		V_j=V(k_j), \qquad j=0,\dots,N,\\
		V(k)^*V(k)=I, \qquad V(k)V(k)^*=P(k), \qquad k\in[-1,1].
	\end{gathered}
	\end{equation}
	\item[(c)] \textbf{Interpolation is spectrally accurate.} The Chebyshev interpolant $\mathcal I_NV$, which by~(b)
	depends only on the computed frames $V_j=V(k_j)$, satisfies
	\begin{equation}\label{eq:main3}
	\begin{gathered}
		\norm{\mathcal I_NV-V}_{[-1,1]}\le C\rho'^{-N},\\
		\norm{(\mathcal I_NV)(\mathcal I_NV)^*-P}_{[-1,1]}\le C\rho'^{-N}, \qquad
		\norm{(\mathcal I_NV)^*(\mathcal I_NV)-I}_{[-1,1]}\le C\rho'^{-N}.
	\end{gathered}
	\end{equation}
\end{itemize}
\end{theorem}

\begin{remark}[The gauge function]\label{rmk:gauge}
The polar factor in Theorem~\ref{thm:main}(a) is needed since the interpolant of the gauges need not be unitary
between the nodes; at the nodes, it equals the gauges, which are already unitary, so $U(k_j)=U_j$. The gauges, and
hence $U$, depend only on the parallel-transported frame at the nodes: $U_{j+1}=\polar(W_{j+1}^*W_j)\,U_j$ with
$U_0=I$ (Lemma~\ref{lem:factor}).
\end{remark}

Theorem~\ref{thm:main}(c) concerns the frame $V=WU$, not the parallel-transported frame $W$. On $[-1,1]$,
\begin{equation}\label{eq:IVminusW}
	\mathcal I_NV-W=(\mathcal I_NV-V)+W(U-I),
\end{equation}
where the first term is $O(\rho'^{-N})$ by Theorem~\ref{thm:main}(c), while the second is analytic and $O(N^{-2})$ by
part~(a) and equals $V_j-W_j$ at the nodes, the difference of Proposition~\ref{prop:second}. Thus the interpolant approaches $W$ only at
second order; the cause is the gauge $U$ in the term $W(U-I)$, not an interpolation error. Corollary~\ref{cor:modified}
describes $U$ through the connection matrix of $V$, and Corollary~\ref{cor:recover} uses this to remove the gauge and
recover $W$ to spectral accuracy.

Theorem~\ref{thm:main} has a differential-equation form: the frame $V$ solves the parallel-transport
equation~\eqref{eq:kato} with one additional term $V\Omega$, where $\Omega$ is the connection matrix of $V$. Indeed,
differentiating $V=WU$ and using $W'=P'W$ gives $V'=P'V+V\,U^*U'$, and $U^*U'=V^*V'$ by~\eqref{eq:conn}. The
computed frames are therefore exact samples of a frame that is parallel-transported for a slightly modified, analytic
connection. This viewpoint is reminiscent of backward error analysis for numerical integrators, which relates a
one-step method to a modified differential equation (Chapter~IX in~\cite{HLW06}). The essential result here, however, is that the analytic frame $V$, which depends on $N$, remains
uniformly bounded on every ellipse $\mathcal E_{\rho'}$, $1<\rho'<\rho$, yielding spectral interpolation at the same
rate as $P$.

\begin{corollary}[Modified parallel-transport equation]\label{cor:modified}
Let $V$ be as in Theorem~\ref{thm:main}, and let $\Omega=V^\sharp V'$. For every
$1<\rho'<\rho$ there are $C$ and $N_0$ such that for every $N\ge N_0$ the function $\Omega$ is analytic on
$\mathcal E_{\rho'}$, equals the connection matrix $V^*V'$ of $V$ on $[-1,1]$, and
\begin{gather}
	V'=P'V+V\Omega \ \text{on }\mathcal E_{\rho'}, \qquad V(-1)=V_0, \label{eq:modkato}\\
	V_j=V(k_j), \qquad j=0,\dots,N, \notag\\
	\norm{\Omega(k)}\le\frac C{N^2}, \qquad k\in\mathcal E_{\rho'} . \label{eq:Omegabound}
\end{gather}
\end{corollary}

When the parallel-transported frame itself is needed, the gauge can be removed with one $n\times n$ linear
differential equation, using only the computed frames. The gauges $U_j=W_j^*V_j$ involve the unknown frame $W$, but their
connection does not: $\Omega=V^*V'=U^*U'$ on $[-1,1]$ by~\eqref{eq:conn}, so the gauge
solves $U'=U\Omega$,
$U(-1)=I$, and $W=VU^*$. In practice, $V$ is replaced by its interpolant $\mathcal I_NV$, which is not exactly a frame
between the nodes, and $\Omega$ by the skew-Hermitian part $\Omega_{\mathcal I}$ of
$(\mathcal I_NV)^*(\mathcal I_NV)'$. Since $\mathcal I_NV$ and its derivative approximate $V$ and $V'$ with geometrically decaying error,
the matrix $\Omega_{\mathcal I}$ and the solution $U_{\mathcal I}$ of the gauge equation with $\Omega_{\mathcal I}$
approximate $\Omega$ and $U$ in the same way.

\begin{corollary}[Recovering the parallel-transported frame]\label{cor:recover}
Let $U$, $V$ and $W$ be as in Theorem~\ref{thm:main} and $\Omega$ as in Corollary~\ref{cor:modified}, let
\begin{equation}\label{eq:OmegaI}
	\Omega_{\mathcal I}=\tfrac12\bigl((\mathcal I_NV)^*(\mathcal I_NV)'-((\mathcal I_NV)')^*(\mathcal I_NV)\bigr)
	\quad\text{on }[-1,1],
\end{equation}
and let $U_{\mathcal I}\colon[-1,1]\to\C^{n\times n}$ be the solution of
\begin{equation}\label{eq:UIeq}
	U_{\mathcal I}'=U_{\mathcal I}\,\Omega_{\mathcal I}, \qquad U_{\mathcal I}(-1)=I.
\end{equation}
Then $U_{\mathcal I}$ is unitary, and for every $1<\rho'<\rho$ there are $C$ and $N_0$ such that for every $N\ge N_0$
\begin{equation}\label{eq:katorec}
\begin{gathered}
	\norm{\Omega_{\mathcal I}-\Omega}_{[-1,1]}\le C\rho'^{-N}, \qquad
	\norm{U_{\mathcal I}-U}_{[-1,1]}\le C\rho'^{-N},\\
	\norm{(\mathcal I_NV)\,U_{\mathcal I}^*-W}_{[-1,1]}\le C\rho'^{-N}.
\end{gathered}
\end{equation}
\end{corollary}

In practice, we solve~\eqref{eq:UIeq} by Picard iteration on its integral form,
\begin{equation}\label{eq:picard}
	U^{(\nu+1)}(k)=I+\int_{-1}^kU^{(\nu)}(t)\,\Omega_{\mathcal I}(t)\,\D t, \qquad U^{(0)}=I.
\end{equation}
By the triangle inequality with~\eqref{eq:Omegabound} and~\eqref{eq:katorec},
$\norm{\Omega_{\mathcal I}}_{[-1,1]}=O(N^{-2})$. Since $U_{\mathcal I}$ is unitary, the classical Picard estimate gives
\begin{equation}\label{eq:picardU}
	\norm{U^{(\nu)}-U_{\mathcal I}}_{[-1,1]}\le\frac{\bigl(2\norm{\Omega_{\mathcal I}}_{[-1,1]}\bigr)^{\nu+1}}{(\nu+1)!},
	\qquad \nu=0,1,\dots,
\end{equation}
and, together with the bound on $W$ in~\eqref{eq:katorec}, this implies
\begin{equation}\label{eq:picardbound}
	\norm{(\mathcal I_NV)\bigl(U^{(\nu)}\bigr)^*-W}_{[-1,1]}\le C\rho'^{-N}+CN^{-2(\nu+1)}, \qquad \nu=0,1,\dots,
\end{equation}
so each iteration reduces the error by a factor $O(N^{-2})$ until it reaches the interpolation error $C\rho'^{-N}$. The integrals are computed by spectral integration on the nodes. For $n=1$ the solution is explicit,
$U_{\mathcal I}=\exp\int_{-1}^k\Omega_{\mathcal I}$.

\section{Estimates for matrix functions}\label{sec:estimates}

This section collects the estimates used in the proofs. Most are standard facts for scalar functions that hold for
matrix functions in the spectral norm, with the same proof. For a continuous function $F$ on a
piecewise smooth curve $\gamma\subset\C$ with values in $\C^{p\times q}$,
\begin{equation}\label{eq:intnorm}
	\Bigl\|\int_\gamma F(\zeta)\,\D\zeta\Bigr\|\le\operatorname{length}(\gamma)\,\norm F_\gamma ,
\end{equation}
since $\abs{u^*(\int_\gamma F\,\D\zeta)v}=\abs{\int_\gamma u^*Fv\,\D\zeta}\le\operatorname{length}(\gamma)\,\norm F_\gamma$
for unit vectors $u\in\C^p$ and $v\in\C^q$. Hence a bound for scalar analytic functions whose proof estimates a contour
integral holds for matrix functions with $\abs\cdot$ replaced by $\norm\cdot$ and the same constant.

\subsection{Analytic matrix functions}\label{sec:analytic}

The following lemma is the Schwarz lemma (Lemma~2.1 in Chapter~8 of~\cite{SteinShakarchi03}) for a zero of order $q$ and
for matrix functions, with the same proof.

\begin{lemma}[Schwarz lemma]\label{lem:schwarz}
Let $X$ be an analytic function on the disk $\{h\in\C:\abs h<r\}$ with values in $\C^{n\times n}$ and
$\norm{X(h)}\le1$ there, and suppose that $X(h)=O(h^q)$ as $h\to0$ for an integer $q\ge1$. Then
\begin{equation}\label{eq:schwarz}
	\norm{X(h)}\le\Bigl(\frac{\abs h}r\Bigr)^q, \qquad \abs h<r.
\end{equation}
\end{lemma}

\begin{proof}
For unit vectors $u,v\in\C^n$, the function $u^*X(h)v/h^q$ is analytic on the disk, since its singularity at $0$ is
removable, and it is bounded by $s^{-q}$ on each circle $\abs h=s<r$, hence on $\abs h\le s$ by the maximum modulus
principle. Letting $s\to r$ gives $\abs{u^*X(h)v}\le(\abs h/r)^q$, and the supremum over $u$ and $v$
is~\eqref{eq:schwarz}.
\end{proof}

The next lemma bounds the Chebyshev interpolation error and its derivative uniformly over the class of analytic
functions bounded by $B$ on an ellipse. The uniformity matters since the frames interpolated in
Theorem~\ref{thm:main} change with $N$.

\begin{lemma}[Chebyshev interpolation of analytic functions]\label{lem:cheb}
Let $1<\rho'<\rho''$. There is $C$, depending only on $\rho'$ and $\rho''$, such that for every analytic matrix
function $F$ on $\mathcal E_{\rho''}$ with $\norm F_{\mathcal E_{\rho''}}\le B$ and every $N\ge1$
\begin{equation}\label{eq:chebbounds}
	\norm{F-\mathcal I_NF}_{[-1,1]}\le CB\rho'^{-N}, \qquad \norm{F'-(\mathcal I_NF)'}_{[-1,1]}\le CB\rho'^{-N}.
\end{equation}
\end{lemma}

\begin{proof}
All theorem and equation numbers in this proof refer to~\cite{ATAP}. The first bound is~(8.3) in Theorem~8.2, with
constant $4/(\rho''-1)$. The second is Theorem~21.1, stated for a fixed function as $O(\rho'^{-N})$; its proof uses
$F$ only through the bound $2B\rho''^{-l}$ of Theorem~8.1 on the Chebyshev coefficients, so its constant is
proportional to $B$. Both bounds hold for matrix
functions by~\eqref{eq:intnorm}, since Theorem~8.1 is proved by estimating a contour integral.
\end{proof}

The next lemma is the analytic continuation of the polar factor~\eqref{eq:polar} of a matrix function close to $I$,
with $X^\sharp$ in place of $X^*$ and the inverse square root given by its binomial series, together with a bound.
For a real matrix that depends analytically on a real variable and is nonsingular, the polar factor is analytic
(Theorem~3 in~\cite{MR93}); the lemma continues it off the real axis. The lemma allows several variables: on an
open set $\mathcal D\subset\C^d$, analytic means that near each point the entries are convergent power series in the
$d$ variables, and $X^\sharp(z)=X(\bar z)^*$, the bar conjugating each coordinate, which is analytic as in
Section~\ref{sec:notation}.

\begin{lemma}[Analytic polar factor]\label{lem:polar}
Let $\mathcal D\subset\C^d$ be open and symmetric about $\R^d$, and let $X\colon\mathcal D\to\C^{n\times n}$ be
analytic with $\norm{X-I}\le\tau\le\frac14$ on $\mathcal D$. Let $a_l$ be the Taylor coefficients of
$(1-t)^{-1/2}=\sum_{l\ge0}a_lt^l$, and
\begin{equation}\label{eq:binom}
	(X^\sharp X)^{-1/2}:=\sum_{l=0}^\infty a_l\,(I-X^\sharp X)^l \qquad\text{on }\mathcal D.
\end{equation}
Then $\polar(X)$ continues analytically to $\mathcal D$ as $X(X^\sharp X)^{-1/2}$, and
\begin{equation}\label{eq:polarext}
	\norm{X(X^\sharp X)^{-1/2}-I}\le4\tau\le1 \qquad\text{on }\mathcal D.
\end{equation}
\end{lemma}

\begin{proof}
Since $\mathcal D$ is symmetric about $\R^d$, $\norm{X^\sharp-I}\le\tau$ on $\mathcal D$ as well. Writing
$X^\sharp X-I=(X^\sharp-I)(X-I)+(X^\sharp-I)+(X-I)$ gives $t:=\norm{X^\sharp X-I}\le\tau^2+2\tau\le\frac94\tau\le\frac9{16}$,
since $\tau\le\frac14$. Hence the series~\eqref{eq:binom} converges uniformly on $\mathcal D$, and $X(X^\sharp X)^{-1/2}$
is analytic there. At a real point, $X^\sharp X=X^*X$ is Hermitian with eigenvalues in $[1-t,1+t]$, so
diagonalizing it shows that the series equals $(X^*X)^{-1/2}$, and $X(X^\sharp X)^{-1/2}=\polar(X)$
by~\eqref{eq:polar}.

For~\eqref{eq:polarext}, the coefficients $a_l$ are positive, so
\begin{equation}\label{eq:invsqrt}
	\norm{(X^\sharp X)^{-1/2}-I}\le\sum_{l\ge1}a_lt^l=(1-t)^{-1/2}-1\le t\le\tfrac94\tau ,
\end{equation}
where $(1-t)^{-1/2}-1=t/\bigl(\sqrt{1-t}\,(1+\sqrt{1-t})\bigr)\le t$ for $0\le t\le\frac9{16}$. Then
$X(X^\sharp X)^{-1/2}-I=(X-I)(X^\sharp X)^{-1/2}+\bigl((X^\sharp X)^{-1/2}-I\bigr)$ has norm at most $\tau(1+\frac94\tau)+\frac94\tau\le4\tau\le1$ by~\eqref{eq:invsqrt}.
\end{proof}

\subsection{Discrete Fourier analysis}\label{sec:dfa}

The substitution $k=-\cos\theta$ passes from Chebyshev interpolation to discrete Fourier analysis. It takes the
equispaced points
\begin{equation}\label{eq:thetaj}
	\theta_j=\frac{\pi j}N, \qquad 0\le j\le2N,
\end{equation}
to the Chebyshev nodes~\eqref{eq:chebnodes}, $k_j=-\cos\theta_j$ for $0\le j\le N$, and a function of $k\in[-1,1]$ to an
even $2\pi$-periodic function of $\theta$; the sign only orders the nodes from $k_0=-1$ to $k_N=1$.

A $2\pi$-periodic matrix function $X$ that is analytic and bounded by $B$ on the strip $\mathcal S_\beta$ has the Fourier
series
\begin{equation}\label{eq:fourier}
	X(\theta)=\sum_{m\in\Z}C_m\,e^{\I m\theta}, \qquad C_m:=\frac1{2\pi}\int_0^{2\pi}X(\theta)\,e^{-\I m\theta}\,\D\theta,
	\qquad \norm{C_m}\le Be^{-\beta\abs m} ,
\end{equation}
where the bound, (B.21) in Theorem~B.1 of~\cite{ATAP}, is proved by shifting the interval of integration into the
strip and so holds for matrix functions by~\eqref{eq:intnorm}.

A sequence $\{X_j\}$ of $2N$ matrices at the equispaced points $\theta_0,\dots,\theta_{2N-1}$ of~\eqref{eq:thetaj} is
indexed modulo $2N$, with $X_{2N}=X_0$, since $\theta_{2N}=2\pi$ and $\theta_0=0$ are the same point of the circle. It has the discrete Fourier coefficients $\hat X_l$, and it is recovered from them by the inversion
formula,
\begin{equation}\label{eq:dft}
	\hat X_l=\frac1{2N}\sum_{j=0}^{2N-1}X_j\,e^{-\I l\theta_j},\ \ -N<l\le N, \quad
	X_j=\sum_{l=-N+1}^{N}\hat X_l\,e^{\I l\theta_j},\ \ 0\le j<2N .
\end{equation}
The coefficients are taken with $-N<l\le N$, so the mode $l=N$, which equals $l=-N$ at the points $\theta_j$, is counted
once. If $X_j=X(\theta_j)$ with $X$ as in~\eqref{eq:fourier}, then the discrete Fourier coefficients satisfy the aliasing
formula
\begin{equation}\label{eq:alias}
	\hat X_l=\sum_{q\in\Z}C_{l+2Nq}, \qquad -N<l\le N ,
\end{equation}
since $e^{\I m\theta_j}=e^{\I l\theta_j}$ whenever $m-l$ is a multiple of $2N$.

A sequence $\{X_j\}$ is even if $X_{2N-j}=X_j$ for $0<j<2N$. Then $\hat X_{-l}=\hat X_l$ for $\abs l<N$
by~\eqref{eq:dft}, the modes $\pm l$ of the inversion formula combine into cosines, and $X_j$ is the value at
$\theta_j$ of the even trigonometric polynomial
\begin{equation}\label{eq:evencheb}
	\hat X_0+2\sum_{0<l<N}\hat X_l\cos l\theta+\hat X_N\cos N\theta .
\end{equation}
Since $\cos l\theta=(-1)^lT_l(-\cos\theta)$, this is a matrix polynomial of degree at most $N$ in $k=-\cos\theta$. It is
therefore the Chebyshev interpolant $\mathcal I_NF$ of any function $F$ with $F(k_j)=X_j$ for $0\le j\le N$,
by~\eqref{eq:interp}.

For $\sigma\ge0$, we measure a sequence $\{X_j\}$ with the coefficients $\hat X_l$ of~\eqref{eq:dft} by the weighted norm
\begin{equation}\label{eq:wnorm}
	\norm{\{X_j\}}_\sigma=\sum_{l=-N+1}^{N}\norm{\hat X_l}\,e^{\sigma\abs l} .
\end{equation}
The following lemma collects four properties of the norm.

\begin{lemma}[Weighted discrete Fourier norm]\label{lem:dft}
Let $\{X_j\}$ and $\{Y_j\}$ be sequences of $2N$ matrices, as in~\eqref{eq:dft}, and let $\sigma\ge0$.
\begin{itemize}
	\item[(i)] If $X_j=X(\theta_j)$ for a $2\pi$-periodic matrix function $X$ that is analytic and bounded by $B$ on
	$\mathcal S_\beta$, and $\sigma<\beta$, then $\norm{\{X_j\}}_\sigma\le KB$, where $K$ depends only on $\beta-\sigma$.
	\item[(ii)] If the products $X_jY_j$ are defined, then $\norm{\{X_jY_j\}}_\sigma\le\norm{\{X_j\}}_\sigma\norm{\{Y_j\}}_\sigma$.
	\item[(iii)] If $Y_j=X_{j+1}-X_j$ for $0\le j<2N$, then
	\begin{equation}\label{eq:summing}
		\norm{\{X_j-\hat X_0\}}_\sigma\le\frac N2\,\norm{\{Y_j\}}_\sigma .
	\end{equation}
	\item[(iv)] If $X_{2N-j}=X_j$ for $0<j<2N$, then the polynomial~\eqref{eq:evencheb} has norm at most
	$\norm{\{X_j\}}_\sigma$ for $\abs{\operatorname{Im}\theta}\le\sigma$.
\end{itemize}
\end{lemma}

\begin{proof}
For~(i), by the aliasing formula~\eqref{eq:alias}, since $\abs l\le\abs{l+2Nq}$ for $-N<l\le N$, and by the bound
in~\eqref{eq:fourier},
\begin{equation}
	\norm{\{X_j\}}_\sigma\le\sum_{m\in\Z}\norm{C_m}\,e^{\sigma\abs m}\le B\sum_{m\in\Z}e^{-(\beta-\sigma)\abs m}
	=B\coth\frac{\beta-\sigma}2 ,
\end{equation}
which is~(i) with $K:=\coth\frac{\beta-\sigma}2$.

For~(ii), inserting the inversion formula in~\eqref{eq:dft} for $X_j$ and $Y_j$ into the formula for the
coefficients shows that the $l$th coefficient of $\{X_jY_j\}$ is the sum of $\hat X_p\hat Y_q$ over the pairs
$-N<p,q\le N$ with $p+q\equiv l\pmod{2N}$. For each such pair, $\abs l\le\abs{p+q}\le\abs p+\abs q$, since $\abs l$ is
the distance from $p+q$ to the nearest multiple of $2N$. Hence
\begin{equation}
	\norm{\{X_jY_j\}}_\sigma\le\sum_{-N<l\le N}\ \sum_{p+q\equiv l}\norm{\hat X_p}e^{\sigma\abs p}\,\norm{\hat Y_q}
	e^{\sigma\abs q}=\norm{\{X_j\}}_\sigma\norm{\{Y_j\}}_\sigma ,
\end{equation}
since the double sum runs over all pairs $-N<p,q\le N$, each exactly once.

For~(iii), \eqref{eq:dft} shows that the coefficients of $\{X_{j+1}\}$ are those of $\{X_j\}$ multiplied by
$e^{\I\pi l/N}$. Hence $\hat Y_l=\bigl(e^{\I\pi l/N}-1\bigr)\hat X_l$ for $-N<l\le N$. For $l\ne0$ the multiplier satisfies $\abs{e^{\I\pi l/N}-1}=2\sin\frac{\pi\abs l}{2N}\ge\frac2N$, since
$\sin x\ge2x/\pi$ for $0\le x\le\pi/2$. Hence $\norm{\hat X_l}\le\frac N2\norm{\hat Y_l}$ for $l\ne0$. By~\eqref{eq:dft},
the sequence $\{X_j-\hat X_0\}$ has the coefficients $\hat X_l$ for $l\ne0$ and $0$ for $l=0$. Summing over $l\ne0$
gives~(iii).

For~(iv), $\abs{\cos l\theta}\le e^{\sigma\abs l}$ for $\abs{\operatorname{Im}\theta}\le\sigma$, and $\hat X_{-l}=\hat X_l$ for $\abs l<N$,
so the norm of~\eqref{eq:evencheb} is at most $\norm{\hat X_0}+\sum_{0<\abs l<N}\norm{\hat X_l}\,e^{\sigma\abs l}
+\norm{\hat X_N}\,e^{\sigma N}=\norm{\{X_j\}}_\sigma$.
\end{proof}

\section{Proofs of the main results}\label{sec:proofs}

This section proves Proposition~\ref{prop:second}, Theorem~\ref{thm:main} and Corollaries~\ref{cor:modified}
and~\ref{cor:recover}. All the proofs work with the gauges $U_j=W_j^*V_j$ of~\eqref{eq:gaugej}. Since $V_j=W_jU_j$,
the error of a computed frame is the deviation of its gauge from $I$, $\norm{V_j-W_j}=\norm{U_j-I}$. The gauges
satisfy the recursion $U_{j+1}=M_jU_j$ with $U_0=I$, whose unitary factors $M_j$ are determined by the
parallel-transported frame at two consecutive nodes and do not involve the computed frames (Lemma~\ref{lem:factor}).
Each gauge is therefore a product of such factors, and the proofs bound these products.

Proposition~\ref{prop:second} needs bounds only at the nodes, and Section~\ref{sec:onestep} proves it directly. Each
factor is within $O(\delta^3)$ of $I$ for a step $\delta$ (Lemma~\ref{lem:step}), and since the factors are unitary,
their deviations from $I$ add up to $O(h^2)$.

As explained before Theorem~\ref{thm:main}, the theorem does not concern a single function. For each $N$, the
discrete parallel transport produces samples of a different function: the gauges, and with them the frame $V$, change
with $N$. Interpolation therefore converges at the rate $\rho'^{-N}$ only if these functions are analytic and bounded
on a fixed ellipse, uniformly in $N$. The gauges, however, are defined only at the nodes, through the recursion, and small values at the nodes do not bound
their interpolant on an ellipse: the values $N^{-2}(-1)^j$ are small, but their interpolant
$\pm N^{-2}T_N$ grows exponentially in $N$ on $\mathcal E_{e^\sigma}$, $\sigma>0$, as $N^{-2}e^{\sigma N}$. The proof therefore carries analyticity
through the recursion itself, in three steps.
\begin{itemize}
	\item[1.] Section~\ref{sec:strip} passes to the circle by the substitution $k=-\cos\theta$. Since the step is
	reversible (Lemma~\ref{lem:rev}), the transport out and back along the path retraces itself, so the gauges form an even sequence that
	satisfies the recursion around the whole circle (Lemma~\ref{lem:retrace}). The one-step discrepancies $M_j-I$ are
	the values at the points $\theta_j$ of a function that is analytic on a fixed strip and of size $O(h^3)$ there
	(Proposition~\ref{prop:strip}).
	\item[2.] Section~\ref{sec:gauge} measures sequences by the weighted norm~\eqref{eq:wnorm}, which plays the role of
	analyticity on a strip: for the values $N^{-2}(-1)^j$ above, it equals $N^{-2}e^{\sigma N}$, the growth of their
	interpolant on $\mathcal E_{e^\sigma}$. The recursion preserves this norm at the cost of one factor $N$
	(Lemma~\ref{lem:running}). The accumulated discrepancies $U_j-I$ are therefore of size $O(h^2)$ in this norm, and the
	interpolant of the gauges is within $O(N^{-2})$ of $I$ on a fixed ellipse (Proposition~\ref{prop:ellipse}).
	\item[3.] The analytic polar factor of this interpolant is the gauge $U$ of Theorem~\ref{thm:main}. It gives the
	analytic frame $V=WU$, which is bounded on a fixed ellipse uniformly in $N$, and Lemma~\ref{lem:cheb} gives the
	spectral accuracy of its interpolant.
\end{itemize}
Section~\ref{sec:corproofs} derives the corollaries from Theorem~\ref{thm:main}: Corollary~\ref{cor:modified} by
differentiating $V=WU$ and a Cauchy estimate, and Corollary~\ref{cor:recover} by comparing the connection matrix of the
interpolant with that of $V$.

Throughout this section, $P$, $J$, $V_0$, $W$ and the nodes are as in Proposition~\ref{prop:second} and
Theorem~\ref{thm:main}, the frames $V_j$ are computed by the discrete parallel transport~\eqref{eq:dt}, and
$W_j=W(k_j)$. For Theorem~\ref{thm:main} and its corollaries, Lemma~\ref{lem:continue}, applied to the interval
$J\cap\mathcal E_\rho$, shows that $W$ extends analytically to $\mathcal E_\rho$ and satisfies~\eqref{eq:katoext} there,
since $\mathcal E_\rho$ is simply connected and symmetric about the real axis.

\subsection{The one-step factor and second-order convergence}\label{sec:onestep}

This section proves Proposition~\ref{prop:second}, on arbitrary nodes. The following lemma shows that each step of the
discrete parallel transport multiplies the gauge by a unitary matrix that is determined by the parallel-transported
frame at the two nodes, not by the computed frames. It reduces Proposition~\ref{prop:second}, and later
Theorem~\ref{thm:main}, to estimates of a single factor. For $n=1$, this
factorization is in the proof of Theorem~2.7 in~\cite{DPP13}.

\begin{lemma}[Factorization]\label{lem:factor}
Under the assumptions of Lemma~\ref{lem:dt}, for $k,k'\in J$ such that $W(k')^*W(k)$ is invertible, let
\begin{equation}\label{eq:M}
	M(k,k')=\polar\bigl(W(k')^*W(k)\bigr)\in\U(n).
\end{equation}
Then, for $0\le j<N$, the matrix $W_{j+1}^*W_j$ is invertible, and the gauges~\eqref{eq:gaugej} satisfy
\begin{equation}\label{eq:factor}
	U_{j+1}=M_jU_j, \qquad M_j:=M(k_j,k_{j+1}), \qquad U_0=I.
\end{equation}
\end{lemma}

\begin{proof}
The matrix $M(k,k')$ is unitary, since it is an $n\times n$ matrix with orthonormal columns.

Fix $0\le j<N$. By Lemma~\ref{lem:dt}(i), the step~\eqref{eq:dt} is defined, so the overlap $W_{j+1}^*V_j$ is
invertible (see before~\eqref{eq:svd}). The SVD form~\eqref{eq:svd} with the frame $Q_{j+1}=W_{j+1}$ at $k_{j+1}$ then gives
\begin{equation}
	V_{j+1}=W_{j+1}\polar\bigl(W_{j+1}^*V_j\bigr).
\end{equation}
Since $V_j=W_jU_j$ by~\eqref{eq:gaugej}, the overlap is $W_{j+1}^*V_j=(W_{j+1}^*W_j)\,U_j$. Since $U_j$ is
unitary, $W_{j+1}^*W_j=(W_{j+1}^*V_j)\,U_j^*$ is invertible, and the second part of~\ref{P} gives
\begin{equation}
	\polar\bigl(W_{j+1}^*V_j\bigr)=\polar\bigl((W_{j+1}^*W_j)\,U_j\bigr)=M_jU_j .
\end{equation}
Hence $V_{j+1}=W_{j+1}M_jU_j$. Multiplying on the left by $W_{j+1}^*$, which satisfies $W_{j+1}^*W_{j+1}=I$, gives
$U_{j+1}=W_{j+1}^*V_{j+1}=M_jU_j$. Finally, $U_0=W_0^*V_0=V_0^*V_0=I$, since $W_0=V_0$ by~\eqref{eq:kato}.
\end{proof}

The next lemma shows that one factor $M(k,k+\delta)$ is the identity up to $O(\delta^3)$. The proofs use only this
bound; the leading term shows that the order is attained whenever $W^*[P',P'']W\ne0$. For $n=1$, the bound is in the
proof of Theorem~2.7 in~\cite{DPP13}, and for a group of eigenvectors of a Hermitian family, the leading term can be
read off from the third-order terms of equations~(38) and~(53) in~\cite{LRS22}.

\begin{lemma}[One step]\label{lem:step}
Let $M$ be as in~\eqref{eq:M}, and suppose that $J$ contains $[-1,1]$. Then we have
\begin{equation}\label{eq:onestep}
	M(k,k+\delta)=I+\frac{\delta^3}{12}\,W(k)^*\bigl[P'(k),P''(k)\bigr]W(k)+O(\delta^4) \qquad\text{as }\delta\to0,
\end{equation}
where $[A,B]=AB-BA$ is the commutator, and the constant in the $O(\delta^4)$ term does not depend on $k\in[-1,1]$.
\end{lemma}

\begin{proof}
Fix $k\in[-1,1]$, and write $W$, $W'$, \dots\ for $W(k)$ and its derivatives. The constants in the $O$-terms below
depend only on bounds for $W$ and its first four derivatives on a compact neighborhood of $[-1,1]$ in $J$, so they are
uniform in $k$. By~\eqref{eq:katoprop}, $W^*W=I$ and $W^*W'=0$, and differentiating the latter gives
$W^*W''=-W'^*W'$ and $W^*W'''=-W''^*W'-2W'^*W''$. Taylor's theorem therefore gives, for $\Delta=W(k+\delta)^*W(k)-I$,
\begin{equation}\label{eq:Delta}
	\Delta=-\frac{\delta^2}2\,W'^*W'-\frac{\delta^3}6\bigl(W'^*W''+2W''^*W'\bigr)+O(\delta^4).
\end{equation}
In particular, $\Delta=O(\delta^2)$, so there is $\delta_*>0$, independent of $k$, such that $k+\delta\in J$ and
$\norm\Delta\le\frac14$ for $\abs\delta\le\delta_*$. Then $M(k,k+\delta)=\polar(I+\Delta)$ is defined by~\eqref{eq:M}.

Since $(I+\Delta)^*(I+\Delta)=I+\Delta+\Delta^*+\Delta^*\Delta$, the first two terms of the series~\eqref{eq:binom},
which converges for $\norm\Delta\le\frac14$ (Lemma~\ref{lem:polar}), give
$\bigl((I+\Delta)^*(I+\Delta)\bigr)^{-1/2}=I-\frac12(\Delta+\Delta^*)+O(\norm\Delta^2)$. Multiplying on the left by
$I+\Delta$ gives
\begin{equation}\label{eq:polarlin}
	\polar(I+\Delta)=I+\tfrac12(\Delta-\Delta^*)+O(\norm\Delta^2),
\end{equation}
with remainder $O(\delta^4)$. The term of order two in~\eqref{eq:Delta} is Hermitian and drops out of $\Delta-\Delta^*$,
so the term of order three gives
\begin{equation}
	M(k,k+\delta)-I=\frac{\delta^3}{12}\bigl(W'^*W''-W''^*W'\bigr)+O(\delta^4).
\end{equation}
Finally, $W'=P'W$ by~\eqref{eq:kato} gives $W''=(P''+P'^2)W$, and $P'$ and $P''$ are Hermitian, so
\begin{equation}
	W'^*W''-W''^*W'=W^*\bigl(P'(P''+P'^2)-(P''+P'^2)P'\bigr)W=W^*[P',P'']W ,
\end{equation}
which gives~\eqref{eq:onestep}.
\end{proof}

\begin{proof}[Proof of Proposition~\ref{prop:second}]
Here $k_0=-1$, and $V_0$ is a frame at $-1$. We first choose the constants. By~\eqref{eq:onestep}, there are
$\delta_0>0$ and $L$ such that $\norm{M(k,k+\delta)-I}\le L\delta^3$ for $k\in[-1,1]$ and $0<\delta\le\delta_0$.
These constants do not depend on $V_0$. Any other frame at $-1$ is $V_0G$ with $G\in\U(n)$ (see after
Definition~\ref{def:frame}), and its parallel-transported frame is $WG$. By~\ref{P}, the factor~\eqref{eq:M} for $WG$
is $G^*M(k,k')G$, which has the same distance from $I$.

We next set $h_0=\min\bigl(\delta_0,1/(1+\norm{P'}_{[-1,1]})\bigr)$ and show that the discrete parallel transport is
defined for nodes with $h\le h_0$. For $0\le i<N$, the mean value inequality gives
$\norm{P(k_{i+1})-P(k_i)}\le h\norm{P'}_{[-1,1]}<1$, since $h\le h_0\le1/(1+\norm{P'}_{[-1,1]})$. Hence the assumptions of Lemma~\ref{lem:dt} hold, the discrete parallel transport is defined, and
Lemma~\ref{lem:factor} applies.

It remains to prove~\eqref{eq:secondorder}. Since $W_j$ has orthonormal columns, $\norm{W_jX}=\norm X$ for
$X\in\C^{n\times n}$, so by~\eqref{eq:gaugej}
\begin{equation}\label{eq:errgauge}
	\norm{V_j-W_j}=\norm{W_j(U_j-I)}=\norm{U_j-I}.
\end{equation}
For $A\in\C^{n\times n}$ and $B\in\U(n)$, $AB-I=(A-I)B+(B-I)$ and $\norm{(A-I)B}=\norm{A-I}$ give
$\norm{AB-I}\le\norm{A-I}+\norm{B-I}$. By Lemma~\ref{lem:factor}, $U_{i+1}=M_iU_i$ with $U_i$ unitary, so
$\norm{U_{i+1}-I}\le\norm{M_i-I}+\norm{U_i-I}$. Since $U_0=I$ by~\eqref{eq:factor}, it follows that
\begin{equation}
	\norm{U_j-I}\le\sum_{i=0}^{j-1}\norm{M_i-I}\le L\sum_{i=0}^{j-1}(k_{i+1}-k_i)^3
	\le Lh^2\sum_{i=0}^{N-1}(k_{i+1}-k_i)=2Lh^2,
\end{equation}
where the second inequality is the bound on $M$ with $k=k_i\in[-1,1]$ and $\delta=k_{i+1}-k_i\le h\le\delta_0$.
By~\eqref{eq:errgauge}, this is~\eqref{eq:secondorder} with $C:=2L$. Since $\delta_0$ and $L$ depend only on
$P$, so do $h_0$ and $C$.
\end{proof}

\subsection{Discrete parallel transport on the circle}\label{sec:strip}

In Theorem~\ref{thm:main}, the discrete parallel transport runs on the Chebyshev nodes~\eqref{eq:chebnodes}. As is
standard in approximation theory, we prove the theorem in Fourier space, through the substitution \mbox{$k=-\cos\theta$} of Section~\ref{sec:dfa}. We define
\begin{equation}\label{eq:ktheta}
	k(\theta)=-\cos\theta, \qquad h=\frac\pi N ,
\end{equation}
so that $h$ is the spacing of the points $\theta_j=jh$ of~\eqref{eq:thetaj}, and $k_j=k(\theta_j)$ for $0\le j\le N$.
We use this formula to define $k_j$ also for $N<j\le2N$. Since $\cos(2\pi-\theta)=\cos\theta$,
\begin{equation}\label{eq:retrace}
	k_{2N-j}=k_j, \qquad 0\le j\le2N ,
\end{equation}
so the substitution wraps the interval around the circle: as $j$ runs from $0$ to $2N$, the node $k_j$ runs from $-1$
to $1$ through the Chebyshev nodes and back. By~\eqref{eq:striptoellipse}, the map $\theta\mapsto k(\theta)$ takes the strip $\mathcal S_s$ onto $\mathcal E_{e^s}$ for every
$s>0$.

The passage from Chebyshev to Fourier (Section~\ref{sec:dfa}) extends functions on $[-1,1]$ evenly around the circle.
The following lemma shows that this extension is compatible with the recursion of the gauges: continued back along the
path over the nodes $k_{N+1},\dots,k_{2N}$, the discrete parallel transport retraces the computed frames, since the step
is reversible (Lemma~\ref{lem:rev}).

\begin{lemma}[Even extension of the gauges]\label{lem:retrace}
Suppose that $\norm{P(k_{j+1})-P(k_j)}<1$ for $0\le j<N$, and let $V_0,\dots,V_{2N}$ be the discrete parallel
transport~\eqref{eq:dt} on the nodes $k_0,\dots,k_{2N}$. Then $V_{2N-j}=V_j$ for $0\le j\le2N$, and the gauges
$U_j=W_j^*V_j$, with $W_j=W(k_j)$, satisfy
\begin{equation}\label{eq:evenrec}
	U_{2N-j}=U_j, \quad 0\le j\le2N, \qquad U_{j+1}=M_jU_j, \quad 0\le j<2N,
\end{equation}
with $U_0=I$ and $M_j=M(k_j,k_{j+1})$ as in~\eqref{eq:factor}.
\end{lemma}

\begin{proof}
We first show that the transport is defined. By~\eqref{eq:retrace}, for $N\le j<2N$ the nodes $k_j$ and $k_{j+1}$ are
$k_{i+1}$ and $k_i$ with $i=2N-j-1\in\{0,\dots,N-1\}$, so $\norm{P(k_{j+1})-P(k_j)}<1$ for $0\le j<2N$. Hence
Lemma~\ref{lem:dt}, with $2N$ in place of $N$, shows that the transport is defined and that every $V_j$ is a frame at
$k_j$. For $0\le j\le N$, the $V_j$ are the frames of Theorem~\ref{thm:main}, since the nodes and the recursion are the
same.

We next show that $V_{N+i}=V_{N-i}$ for $0\le i\le N$, by induction on $i$. For
$i=0$, both sides are $V_N$. Suppose that $V_{N+i}=V_{N-i}$ for some $0\le i<N$. Since $k_{N+i+1}=k_{N-i-1}$
by~\eqref{eq:retrace},
\begin{equation}
	V_{N+i+1}=\polar\bigl(P(k_{N+i+1})\,V_{N+i}\bigr)=\polar\bigl(P(k_{N-i-1})\,V_{N-i}\bigr)=V_{N-i-1},
\end{equation}
where the last equality is~\eqref{eq:rev} with $j=N-i-1$.

It remains to prove~\eqref{eq:evenrec}. For $0\le j\le2N$, \eqref{eq:retrace} and $V_{2N-j}=V_j$ give
\begin{equation*}
	U_{2N-j}=W(k_{2N-j})^*V_{2N-j}=W(k_j)^*V_j=U_j .
\end{equation*}
In particular $U_{2N}=U_0$, consistent with the indexing modulo $2N$ of Section~\ref{sec:dfa}.
Lemma~\ref{lem:factor} on the nodes $k_0,\dots,k_{2N}$ gives $U_{j+1}=M_jU_j$ for $0\le j<2N$, with $U_0=I$.
\end{proof}

With the one-step factor $M$ in~\eqref{eq:M}, we define the one-step discrepancy as a function of the position
$\theta$ and the step $\eta$,
\begin{equation}\label{eq:D}
	D(\theta,\eta)=M\bigl(k(\theta),k(\theta+\eta)\bigr)-I, \qquad \theta,\eta\in\R,
\end{equation}
whenever $W(k(\theta+\eta))^*W(k(\theta))$ is invertible. Since $k(\theta_j)=k_j$ and $\theta_j+h=\theta_{j+1}$, its
values $D(\theta_j,h)=M_j-I$, $0\le j<2N$, are the deviations from $I$ of the factors of the gauge
recursion~\eqref{eq:evenrec}.

The following proposition shows that, for large $N$, $D(\cdot,h)$ extends to a $2\pi$-periodic function that is
analytic on a fixed strip and of size $O(h^3)$ there. The proof carries the bound of Lemma~\ref{lem:step} from the real
axis into the strip, by treating the step as a complex variable.

\begin{proposition}[The one-step discrepancy in a strip]\label{prop:strip}
Let $P$ be analytic on $\mathcal E_\rho$, as in Theorem~\ref{thm:main}, and let $0<\beta'<\log\rho$. There are $N_1$
and $L$ such that for every $N\ge N_1$ the function $D(\cdot,h)$, with $D$ as in~\eqref{eq:D}, is defined on $\R$
and extends to a $2\pi$-periodic analytic function on $\mathcal S_{\beta'}$ with
\begin{equation}\label{eq:Dbound}
	\norm{D(\theta,h)}\le Lh^3, \qquad \theta\in\mathcal S_{\beta'}.
\end{equation}
\end{proposition}

\begin{proof}
We treat the step $\eta$ as a second complex variable. We show that, for some $\eta_1>0$ and the disk
$\mathbb D=\{\eta\in\C:\ \abs\eta<\eta_1\}$, $D$ extends to an analytic function on $\mathcal S_{\beta'}\times\mathbb D$
with $\norm D\le1$, and that $D(\theta,\eta)=O(\eta^3)$ as $\eta\to0$ for every $\theta\in\mathcal S_{\beta'}$. The
Schwarz lemma in $\eta$ then gives $\norm{D(\theta,\eta)}\le(\abs\eta/\eta_1)^3$, and setting $\eta=h$ proves the
proposition.

We first extend $D$. We fix $\eta_0>0$ with $s=\beta'+\eta_0<\log\rho$, so that $\theta+\eta\in\mathcal S_s$ for
$\theta\in\mathcal S_{\beta'}$ and $\abs\eta<\eta_0$, and we continue the overlap $W(k(\theta+\eta))^*W(k(\theta))$ off
the real axis, with $W^\sharp$ in place of $W^*$:
\begin{equation}
	\Delta(\theta,\eta)=W^\sharp\bigl(k(\theta+\eta)\bigr)\,W\bigl(k(\theta)\bigr)-I .
\end{equation}
Since $\theta\mapsto k(\theta)$ takes $\mathcal S_s$ onto $\mathcal E_{e^s}$ by~\eqref{eq:striptoellipse}, whose closure
lies in $\mathcal E_\rho$, where $W$ and $W^\sharp$ are analytic, $\Delta$ is analytic in $(\theta,\eta)$,
$2\pi$-periodic in $\theta$, and bounded by a constant $B>0$. By~\eqref{eq:katoext}, $\Delta(\theta,0)=0$, so
Lemma~\ref{lem:schwarz} applied to $\Delta(\theta,\cdot)/B$ with $r=\eta_0$ and $q=1$ gives
$\norm{\Delta(\theta,\eta)}\le B\abs\eta/\eta_0$. With $\eta_1=\min(\eta_0,\eta_0/(4B))$, therefore,
$\norm\Delta\le\frac14$ on $\mathcal S_{\beta'}\times\mathbb D$. This set is symmetric about $\R^2$, so
Lemma~\ref{lem:polar} with $X=I+\Delta$ and $\tau=\frac14$ shows that
\begin{equation}
	(I+\Delta)\bigl((I+\Delta)^\sharp(I+\Delta)\bigr)^{-1/2}-I
\end{equation}
is analytic on $\mathcal S_{\beta'}\times\mathbb D$, $2\pi$-periodic in $\theta$, and of norm at most $1$. Since
$W^\sharp=W^*$ on the real axis, Lemma~\ref{lem:polar} shows that this function equals $D$ for real $\theta$ and $\eta$,
so it is the analytic continuation of $D$ to $\mathcal S_{\beta'}\times\mathbb D$, which we denote again by $D$.

We next show that $D$ has a zero of order three in $\eta$. Fix a real $\theta$. For real $\eta$,
$k(\theta+\eta)=k(\theta)+\delta$ with $\abs\delta=\abs{\cos\theta-\cos(\theta+\eta)}\le\abs\eta$, so
$D(\theta,\eta)=O(\eta^3)$ by Lemma~\ref{lem:step}. The Taylor coefficients of $D(\theta,\cdot)$ at $\eta=0$ are its
derivatives along the real axis, so those of orders $0$, $1$ and $2$ vanish. As functions of
$\theta\in\mathcal S_{\beta'}$, these coefficients are analytic and vanish for real $\theta$, so they vanish on
$\mathcal S_{\beta'}$ by the identity theorem. Hence Lemma~\ref{lem:schwarz} applied to $D(\theta,\cdot)$ with $r=\eta_1$ and $q=3$ gives
\begin{equation}\label{eq:Dschwarz}
	\norm{D(\theta,\eta)}\le\Bigl(\frac{\abs\eta}{\eta_1}\Bigr)^3, \qquad (\theta,\eta)\in\mathcal S_{\beta'}\times\mathbb D.
\end{equation}
For $N\ge N_1$, with $N_1>\pi/\eta_1$, the real step $h=\pi/N$ lies in $\mathbb D$, so $D(\cdot,h)$ is defined on $\R$,
and its extension is $2\pi$-periodic and analytic on $\mathcal S_{\beta'}$. With $\eta=h$, \eqref{eq:Dschwarz}
is~\eqref{eq:Dbound} with $L=\eta_1^{-3}$.
\end{proof}


\subsection{Proof of Theorem~\ref{thm:main}}\label{sec:gauge}

This section passes from the one-step discrepancies of Section~\ref{sec:strip} to the gauges~\eqref{eq:gaugej}, which
they build up through the recursion~\eqref{eq:evenrec}, and completes the proof of Theorem~\ref{thm:main}. By
Lemma~\ref{lem:retrace}, the gauges extend to an even sequence $U_0,\dots,U_{2N-1}$, indexed modulo $2N$, and we denote
\begin{equation}\label{eq:Rj}
	U_j=I+R_j, \qquad 0\le j<2N.
\end{equation}
Since $\norm{V_j-W_j}=\norm{R_j}$ by~\eqref{eq:errgauge}, we call
$R_j$ the accumulated discrepancy. By~\eqref{eq:evenrec} and~\eqref{eq:D},
$U_{j+1}=(I+D(\theta_j,h))U_j$, that is,
\begin{equation}\label{eq:discrep}
	R_{j+1}-R_j=D(\theta_j,h)\,(I+R_j), \qquad 0\le j<2N, \qquad R_0=0 .
\end{equation}
The following lemma shows that the weighted norm survives a recursion of this form, at the cost of a factor $2N$.

\begin{lemma}[Running products]\label{lem:running}
Let $\sigma\ge0$, and let $\{X_j\}$ and $\{Y_j\}$ be sequences of $2N$ matrices in $\C^{n\times n}$, as
in~\eqref{eq:dft}, with $X_0=0$ and
\begin{equation}\label{eq:runrec}
	X_{j+1}-X_j=Y_j\,(I+X_j), \qquad 0\le j<2N ,
\end{equation}
where $X_{2N}=X_0$, so that the recursion closes up around the circle at $j=2N-1$. If
$N\norm{\{Y_j\}}_\sigma\le\frac12$, then
\begin{equation}\label{eq:running}
	\norm{\{X_j\}}_\sigma\le2N\norm{\{Y_j\}}_\sigma .
\end{equation}
\end{lemma}

\begin{proof}
Lemma~\ref{lem:dft}(iii) bounds $\{X_j\}$ only up to its mean $\hat X_0$, so we first bound the mean. Since $X_0=0$, the
inversion formula in~\eqref{eq:dft} at $j=0$ gives $\hat X_0=-\sum_{l\ne0}\hat X_l$. As in the proof of
Lemma~\ref{lem:dft}(iii), the sequence $\{X_j-\hat X_0\}$ has the coefficients $\hat X_l$ for $l\ne0$ and $0$ for
$l=0$. Hence
\begin{equation}
	\norm{\hat X_0}\le\sum_{l\ne0}\norm{\hat X_l}\,e^{\sigma\abs l}=\norm{\{X_j-\hat X_0\}}_\sigma ,
\end{equation}
so $\norm{\{X_j\}}_\sigma=\norm{\hat X_0}+\norm{\{X_j-\hat X_0\}}_\sigma\le2\norm{\{X_j-\hat X_0\}}_\sigma$. We next bound $\{X_j-\hat X_0\}$ by applying Lemma~\ref{lem:dft}(iii) to the recursion~\eqref{eq:runrec} and then
Lemma~\ref{lem:dft}(ii):
\begin{equation}
	2\norm{\{X_j-\hat X_0\}}_\sigma\le N\norm{\{Y_j(I+X_j)\}}_\sigma\le N\norm{\{Y_j\}}_\sigma\,\norm{\{I+X_j\}}_\sigma .
\end{equation}
By~\eqref{eq:dft}, adding $I$ to every $X_j$ changes only the coefficient with $l=0$, from $\hat X_0$ to $I+\hat X_0$.
Hence $\norm{\{I+X_j\}}_\sigma\le1+\norm{\{X_j\}}_\sigma$, and the three bounds combine to
\begin{equation*}
	\norm{\{X_j\}}_\sigma\le2\norm{\{X_j-\hat X_0\}}_\sigma\le N\norm{\{Y_j\}}_\sigma\bigl(1+\norm{\{X_j\}}_\sigma\bigr)
	\le N\norm{\{Y_j\}}_\sigma+\tfrac12\norm{\{X_j\}}_\sigma ,
\end{equation*}
using $N\norm{\{Y_j\}}_\sigma\le\frac12$ in the last step, which gives~\eqref{eq:running}.
\end{proof}

The next proposition shows that the Chebyshev interpolant~\eqref{eq:interp} of the gauges~\eqref{eq:gaugej} is within
$C/N^2$ of $I$ on a fixed ellipse,
uniformly in $N$.

\begin{proposition}[The accumulated discrepancy on an ellipse]\label{prop:ellipse}
Let $P$ be analytic on $\mathcal E_\rho$, as in Theorem~\ref{thm:main}, and let $1<\rho''<\rho$. There are $C$ and $N_2$
such that for every $N\ge N_2$ the discrete parallel transport on the Chebyshev nodes~\eqref{eq:chebnodes} is defined,
and the Chebyshev interpolant of the gauges, the matrix polynomial $\mathcal I_NU$ with $(\mathcal I_NU)(k_j)=U_j$ for
$0\le j\le N$, satisfies
\begin{equation}\label{eq:IUbound}
	\norm{(\mathcal I_NU)(k)-I}\le\frac C{N^2}, \qquad k\in\mathcal E_{\rho''} .
\end{equation}
\end{proposition}

\begin{proof}
The proof combines four estimates: Proposition~\ref{prop:strip} and Lemma~\ref{lem:dft}(i) bound the one-step
discrepancies in the weighted norm, Lemma~\ref{lem:running} passes to the accumulated ones, and Lemma~\ref{lem:dft}(iv)
bounds their interpolant on the ellipse.

Let $\sigma=\log\rho''$, and fix $\beta'$ with $\sigma<\beta'<\log\rho$. Proposition~\ref{prop:strip} for this $\beta'$
gives $N_1$ and $L$, and Lemma~\ref{lem:dft}(i) gives $K$. We take $N_2\ge N_1$ such that $h=\pi/N$ satisfies
$h\norm{P'}_{[-1,1]}<1$ and $\pi KLh^2\le\frac12$ for $N\ge N_2$. Since the largest spacing of the Chebyshev nodes is at
most $h$ (see after~\eqref{eq:chebnodes}), the mean value inequality gives
$\norm{P(k_{j+1})-P(k_j)}\le h\norm{P'}_{[-1,1]}<1$ for $0\le j<N$. Hence the transport is defined, and Lemma~\ref{lem:retrace}
and~\eqref{eq:discrep} apply.

By Proposition~\ref{prop:strip} and Lemma~\ref{lem:dft}(i),
\begin{equation}
	\norm{\{D(\theta_j,h)\}}_\sigma\le KLh^3 .
\end{equation}
Since $NKLh^3=\pi KLh^2\le\frac12$, Lemma~\ref{lem:running} applies to the recursion~\eqref{eq:discrep},
$R_{j+1}-R_j=D(\theta_j,h)(I+R_j)$ with $R_0=0$, and gives
\begin{equation}\label{eq:Rbound}
	\norm{\{R_j\}}_\sigma\le2NKLh^3=2\pi KLh^2=\frac{2\pi^3KL}{N^2} .
\end{equation}

Since $U_j=I+R_j$, we have $\mathcal I_NU-I=\mathcal I_NR$. Since $\{R_j\}$ is even by~\eqref{eq:evenrec}, $(\mathcal I_NR)(k(\theta))$ is
given by~\eqref{eq:evencheb} with $X_j=R_j$, so Lemma~\ref{lem:dft}(iv) and~\eqref{eq:Rbound} give
\begin{equation}
	\norm{(\mathcal I_NU)(k(\theta))-I}=\norm{(\mathcal I_NR)(k(\theta))}\le\frac{2\pi^3KL}{N^2}, \qquad
	\abs{\operatorname{Im}\theta}\le\sigma .
\end{equation}
Since $\theta\mapsto k(\theta)$ takes $\mathcal S_\sigma$ onto $\mathcal E_{\rho''}$ by~\eqref{eq:striptoellipse}, this
is~\eqref{eq:IUbound} with $C=2\pi^3KL$.
\end{proof}

\begin{proof}[Proof of Theorem~\ref{thm:main}]
We fix $\rho''$ with $\rho'<\rho''<\rho$. Proposition~\ref{prop:ellipse} gives $C_1$ and $N_2$ such that, for
$N\ge N_2$, the transport is defined and $\norm{(\mathcal I_NU)(k)-I}\le C_1/N^2$ for $k\in\mathcal E_{\rho''}$. We take
$N_0\ge N_2$ with $C_1/N_0^2\le\frac14$, and let $N\ge N_0$.

For~(a), Lemma~\ref{lem:polar} on $\mathcal D=\mathcal E_{\rho''}$, with $X=\mathcal I_NU$ and $\tau=C_1/N^2\le\frac14$,
extends $U=\polar(\mathcal I_NU)$ analytically to $\mathcal E_{\rho''}$, with
\begin{equation}\label{eq:Ubound}
	\norm{U(k)-I}\le\frac{4C_1}{N^2}, \qquad k\in\mathcal E_{\rho''} .
\end{equation}
At the nodes, $(\mathcal I_NU)(k_j)=U_j$ is unitary, so $U(k_j)=\polar(U_j)=U_j$. This proves~(a) with $C=4C_1$.

For~(b), $V=WU$ is analytic on $\mathcal E_{\rho''}$ by Lemma~\ref{lem:continue}, and $V(k_j)=W_jU_j=V_j$
by~\eqref{eq:gaugej}. On $[-1,1]$, $W$ is a frame and $U$ is unitary, so $V^*V=I$ and $VV^*=WW^*=P$.

For~(c), let $B$ bound $\norm W$ on $\mathcal E_{\rho''}$, whose closure lies in $\mathcal E_\rho$. By~\eqref{eq:Ubound},
$\norm V\le B(1+4C_1/N^2)\le2B$ on $\mathcal E_{\rho''}$, so Lemma~\ref{lem:cheb}, with its constant $C_2$ for $\rho'$
and $\rho''$, gives for $E=\mathcal I_NV-V$
\begin{equation}\label{eq:Ebound}
	\norm{E}_{[-1,1]}\le C_3\rho'^{-N}, \qquad C_3=2C_2B .
\end{equation}
By~(b), on $[-1,1]$,
\begin{equation}
\begin{gathered}
	(\mathcal I_NV)(\mathcal I_NV)^*-P=EV^*+VE^*+EE^*,\\
	(\mathcal I_NV)^*(\mathcal I_NV)-I=E^*V+V^*E+E^*E,
\end{gathered}
\end{equation}
and since $\norm V=1$ there, both have norm at most $(2+\norm E)\norm E\le(2+C_3)C_3\rho'^{-N}$. This proves~(c), and
the theorem holds with $C=\max\bigl(4C_1,(2+C_3)C_3\bigr)$.
\end{proof}

\subsection{Proofs of the corollaries}\label{sec:corproofs}

This section proves Corollaries~\ref{cor:modified} and~\ref{cor:recover}. Both apply Theorem~\ref{thm:main} on a
slightly larger ellipse.

\begin{proof}[Proof of Corollary~\ref{cor:modified}]
We fix $\rho''$ with $\rho'<\rho''<\rho$. Theorem~\ref{thm:main}(a) on $\mathcal E_{\rho''}$ gives $C_1$ and $N_0$ such
that, for $N\ge N_0$, $U$ is analytic on $\mathcal E_{\rho''}$ and $\norm{U(k)-I}\le C_1/N^2$ there. We take $N_0$ with
$C_1/N_0^2\le1$.

Differentiating $V=WU$ and using $W'=P'W$ and~\eqref{eq:katoext} gives
\begin{equation}
	\Omega=V^\sharp V'=U^\sharp W^\sharp(P'WU+WU')=U^\sharp U' \qquad\text{on }\mathcal E_{\rho''} ,
\end{equation}
so $\Omega$ is analytic there and equals $V^*V'$ on $[-1,1]$. On $[-1,1]$, where $U$ is unitary,
$V'=P'WU+WU'=P'V+VU^*U'=P'V+V\Omega$, and the identity theorem extends this to $\mathcal E_{\rho''}$. Also,
$V(-1)=W(-1)U(-1)=V_0$, since $U(k_0)=U_0=I$, and $V_j=V(k_j)$ by Theorem~\ref{thm:main}(b).

For~\eqref{eq:Omegabound}, let $r>0$, depending only on $\rho'$ and $\rho''$, be such that the circle of radius $r$ about
any point of $\mathcal E_{\rho'}$ lies in $\mathcal E_{\rho''}$. The Cauchy integral formula for $(U-I)'$ on these circles
and~\eqref{eq:intnorm} give $\norm{U'}\le C_1/(rN^2)$ on $\mathcal E_{\rho'}$. Since $\mathcal E_{\rho''}$ is symmetric
about the real axis, $\norm{U^\sharp-I}\le1$ there as well, so $\norm{U^\sharp}\le2$ and
\begin{equation}
	\norm{\Omega(k)}\le\norm{U^\sharp(k)}\,\norm{U'(k)}\le\frac{2C_1}{rN^2}, \qquad k\in\mathcal E_{\rho'} ,
\end{equation}
which is~\eqref{eq:Omegabound} with $C=2C_1/r$.
\end{proof}

\begin{proof}[Proof of Corollary~\ref{cor:recover}]
Let $\rho''$, $B$ and $C_3$ be as in the proof of Theorem~\ref{thm:main}(c), and let $E=\mathcal I_NV-V$, the
interpolation error from which all three bounds in~\eqref{eq:katorec} come. Lemma~\ref{lem:cheb} gives
\begin{equation}\label{eq:EEbound}
	\norm{E}_{[-1,1]}\le C_3\rho'^{-N}, \qquad \norm{E'}_{[-1,1]}\le C_3\rho'^{-N}.
\end{equation}
We take $N_0$ large enough that~\eqref{eq:Omegabound} gives $\norm\Omega\le1$ on $[-1,1]$; then, since $\norm V=1$
there, \eqref{eq:modkato} gives $\norm{V'}_{[-1,1]}\le\norm{P'}_{[-1,1]}+1=:C_4$.

The unitarity of $U_{\mathcal I}$ and the bound on $U_{\mathcal I}-U$ both follow from one identity: if $A'=AX$ and
$B'=BY$ on $[-1,1]$ with $Y$ skew-Hermitian, then the product rule gives
\begin{equation}\label{eq:prodrule}
	(AB^*)'=A'B^*+A(B')^*=A(X-Y)B^* .
\end{equation}
With $A=B=U_{\mathcal I}$ and $X=Y=\Omega_{\mathcal I}$, which is skew-Hermitian by~\eqref{eq:OmegaI}, the
equation~\eqref{eq:UIeq} gives $(U_{\mathcal I}U_{\mathcal I}^*)'=0$, so
$U_{\mathcal I}U_{\mathcal I}^*=U_{\mathcal I}(-1)U_{\mathcal I}(-1)^*=I$ on $[-1,1]$.

We now prove the three bounds in~\eqref{eq:katorec}. For $\Omega_{\mathcal I}-\Omega$, note that $\Omega=V^*V'$ is skew-Hermitian on $[-1,1]$, and
$\Omega_{\mathcal I}$ is the skew-Hermitian part of $(\mathcal I_NV)^*(\mathcal I_NV)'=(V+E)^*(V'+E')$, so
$\Omega_{\mathcal I}-\Omega$ is the skew-Hermitian part of $E^*V'+V^*E'+E^*E'$, and by~\eqref{eq:EEbound}
\begin{equation}\label{eq:OmegaIbound}
	\norm{\Omega_{\mathcal I}-\Omega}_{[-1,1]}\le C_4\norm E_{[-1,1]}+\norm{E'}_{[-1,1]}+\norm E_{[-1,1]}\norm{E'}_{[-1,1]}
	\le C_5\rho'^{-N},
\end{equation}
with $C_5=(C_4+1+C_3)C_3$. For $U_{\mathcal I}-U$, we have $U'=U\Omega$ on $[-1,1]$, since $\Omega=U^*U'$ there by~\eqref{eq:conn},
so~\eqref{eq:prodrule} with $A=U_{\mathcal I}$ and $B=U$ gives $(U_{\mathcal I}U^*)'=U_{\mathcal I}(\Omega_{\mathcal I}-\Omega)U^*$.
Integrating from $-1$, where $U_{\mathcal I}U^*=I$, and using that the interval has length at most $2$ and that
$U_{\mathcal I}$ and $U$ are unitary, so that $\norm{U_{\mathcal I}-U}=\norm{U_{\mathcal I}U^*-I}$, we get
\begin{equation}\label{eq:UIminusU}
	\norm{U_{\mathcal I}-U}_{[-1,1]}\le2\norm{\Omega_{\mathcal I}-\Omega}_{[-1,1]}\le2C_5\rho'^{-N}.
\end{equation}
For the recovered frame, $W=VU^*$ on $[-1,1]$, so
$(\mathcal I_NV)\,U_{\mathcal I}^*-W=(V+E)\,U_{\mathcal I}^*-VU^*=EU_{\mathcal I}^*+V(U_{\mathcal I}-U)^*$, whose norm is
at most $\norm E+\norm{U_{\mathcal I}-U}\le3C_5\rho'^{-N}$ by~\eqref{eq:EEbound} and~\eqref{eq:UIminusU}, since
$C_3\le C_5$. This proves~\eqref{eq:katorec} with $C=3C_5$.
\end{proof}

\section{Numerical examples}\label{sec:numerics}

This section illustrates Proposition~\ref{prop:second}, Theorem~\ref{thm:main} and Corollary~\ref{cor:recover}.
Examples~1 and~2 are Hermitian families whose eigenvalues nearly cross or cross, Example~3 is a non-normal family with
constant eigenvalues, Example~4 is a non-normal family whose two eigenvalues in the subspace have a branch point, and
Example~5 is a large Hermitian family with a spectral gap. Figure~\ref{fig:eigs} shows the eigenvalues of
Examples~1, 2, 4 and~5.

In each example, the frames $Q_j$ at the Chebyshev nodes~\eqref{eq:chebnodes} are eigenvectors computed by the LAPACK
routine \texttt{zheev}, in the gauge that it returns; for Example~3, they are eigenvectors computed by
\texttt{zgeev} and normalized, and for Example~4, the first two Schur vectors computed by \texttt{zgees}. The discrete parallel transport is computed by Algorithm~\ref{alg:dt} from
$V_0=W(-1)$. The interpolant $\mathcal I_NV$ and its derivative are evaluated from their
Chebyshev coefficients. For Corollary~\ref{cor:recover}, $\Omega_{\mathcal I}$ is evaluated at the nodes, and the
Picard iteration~\eqref{eq:picard} is carried out on the nodes with Chebyshev spectral integration, until the update is at the level of machine epsilon; this takes at most $15$ iterations, and at most $7$ for $N\ge64$. We use the iteration also for $n=1$, despite
the explicit solution $U_{\mathcal I}=\exp\int_{-1}^k\Omega_{\mathcal I}$. Since $U^{(0)}=I$, the error after $\nu=0$ iterations is
$\norm{\mathcal I_NV-W}$; Figure~\ref{fig:ex4}(c) shows it for Example~4.

Errors are measured in the spectral norm~\eqref{eq:norms}, and the maximum over $[-1,1]$ is taken over $1000$ equispaced points ($100$ in Example~5), none of which is a node.
Example~1 is computed in quadruple precision, so that Figure~\ref{fig:ex1} shows the decay below the rounding errors of
double precision; the other examples are computed in double precision.

\begin{figure}[t]
\centering
\includegraphics[width=\textwidth]{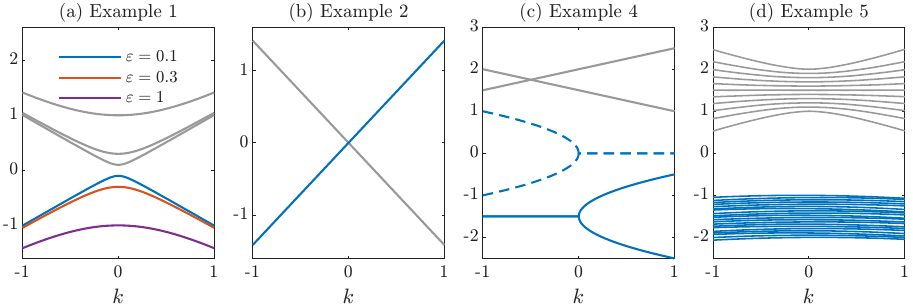}
\caption{The eigenvalues of the families of Examples~1, 2, 4 and~5, in color those whose invariant subspace is
$\Ran P(k)$, in gray the others. (a) Example~1, the avoided crossing~\eqref{eq:ex1}, for $\varepsilon=0.1$, $0.3$
and $1$. (b) Example~2, the exact crossing~\eqref{eq:ex2}; $P(k)$ follows the branch $sk$ through the crossing.
(c) Example~4, the family~\eqref{eq:ex4}, real parts solid and imaginary parts dashed; $P(k)$ projects onto the
invariant subspace of $-\frac32\pm\sqrt k$, a complex pair for $k<0$.
(d) Example~5, the family~\eqref{eq:ex5}; only the $20$ lowest eigenvalues and $11$ of the other $980$ are shown.
Example~3 is not shown: its eigenvalues are $1$ and $-1$ for every $k$.}
\label{fig:eigs}
\end{figure}

\subsection{Example 1: An avoided crossing}\label{sec:ex1}

The first example is an avoided crossing of two eigenvalues. For a parameter $\varepsilon>0$, the Hermitian family $H$
in~\eqref{eq:ex1} has the eigenvalues $\pm\lambda(k)$, $\lambda(k)=\sqrt{k^2+\varepsilon^2}$. They are closest at $k=0$, where they are $2\varepsilon$ apart. We take $P(k)$ to be the projector
onto the eigenvector of the lower eigenvalue $-\lambda(k)$, so $n=1$ and
\begin{equation}\label{eq:ex1}
	H(k)=\begin{pmatrix} k & \varepsilon e^{-\I k}\\ \varepsilon e^{\I k} & -k\end{pmatrix}, \qquad
	P(k)=\frac12\Bigl(I-\frac{H(k)}{\lambda(k)}\Bigr).
\end{equation}
The projector is analytic except at the branch points $\pm\I\varepsilon$ of $\lambda$, where
the eigenvalues meet. The largest Bernstein ellipse~\eqref{eq:ellipse} that excludes them has semi-minor axis
$\varepsilon$, so $\rho=\varepsilon+\sqrt{1+\varepsilon^2}$, and a narrower avoided crossing gives a slower rate. The
parallel-transported frame started from $V_0=W(-1)$ is
\begin{equation}\label{eq:ex1W}
\begin{gathered}
	W(k)=e^{\I\gamma(k)}\begin{pmatrix}-e^{-\I k}\sin\frac{\varphi(k)}2\\[2pt]\cos\frac{\varphi(k)}2\end{pmatrix},\\
	\cos\varphi(k)=\frac k{\lambda(k)}, \qquad \gamma(k)=\frac{k-\lambda(k)+1+\lambda(-1)}2,
\end{gathered}
\end{equation}
with $\varphi(k)\in(0,\pi)$. We take $\varepsilon=0.1$, $0.3$ and $1$, for which
$\rho=1.105$, $1.344$ and $2.414$. The smallest singular value $\sigma_j$ of the overlaps (Section~\ref{sec:compute}) is at
least $0.78$ in all runs.

\begin{figure}[t]
\centering
\includegraphics[width=\textwidth]{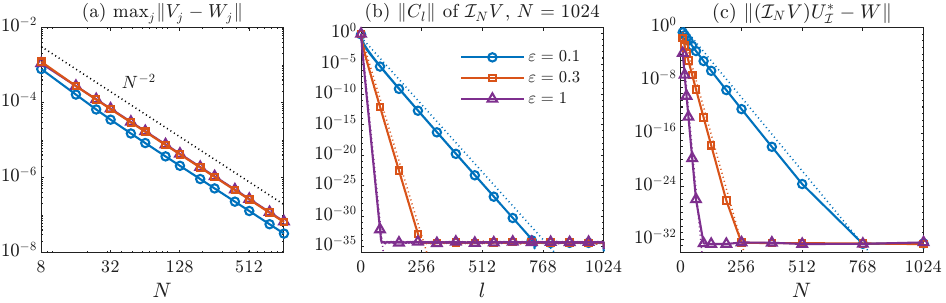}
\caption{Example 1, the avoided crossing~\eqref{eq:ex1}, for $\varepsilon=0.1$, $0.3$ and $1$, in quadruple precision. (a) The error of the computed
frames at the nodes, with $N^{-2}$ (dotted). (b) The Chebyshev coefficients $C_l$ of the interpolant
$\mathcal I_NV$ of the computed frames for $N=1024$, with $\rho^{-l}$ (dotted). (c) The error of the
parallel-transported frame recovered by Corollary~\ref{cor:recover}, with $\rho^{-N}$ (dotted).}
\label{fig:ex1}
\end{figure}

Figure~\ref{fig:ex1}(a) shows that the error $V_j-W_j$ of the computed frames decreases as $N^{-2}$, as
Proposition~\ref{prop:second} states. Figure~\ref{fig:ex1}(b) shows the Chebyshev coefficients of the interpolant
$\mathcal I_NV$ of the computed frames for $N=1024$. They decay at the rate $\rho^{-l}$ down to rounding errors, about
$10^{-35}$, with no plateau at the level of the errors $V_j-W_j$ in~(a). The data $V_j$ thus differ from the samples $W_j$ of the
parallel-transported frame by $O(N^{-2})$, yet their interpolant converges like that of a function analytic on
$\mathcal E_\rho$, since the computed frames are themselves samples of the analytic frame $V$ of
Theorem~\ref{thm:main}(b).

Figure~\ref{fig:ex1}(c) shows that the frame recovered by Corollary~\ref{cor:recover} converges to $W$ at the rate
$\rho^{-N}$ down to rounding errors, about $10^{-32}$. Theorem~\ref{thm:main} and Corollary~\ref{cor:recover} give the rate
$\rho'^{-N}$ for every $\rho'<\rho$, consistent with the observed limiting rate $\rho^{-N}$. By contrast, $\norm{\mathcal I_NV-W}$ stays at the size of
$\max_j\norm{V_j-W_j}$ in Figure~\ref{fig:ex1}(a), due to the term $W(U-I)$ in~\eqref{eq:IVminusW}.

\begin{table}[t]
\caption{Example 1, the avoided crossing~\eqref{eq:ex1}: $N^2\max\norm{U-I}$ on $[-1,1]$ and on the ellipses
$\mathcal E_{1.1}$ and $\mathcal E_{1.2}$, in quadruple precision.}
\label{tab:ex1}
\centering
\begin{tabular}{rcccccc}
\toprule
 & \multicolumn{3}{c}{$\varepsilon=0.3$, $\rho=1.344$} & \multicolumn{3}{c}{$\varepsilon=1$, $\rho=2.414$}\\
\cmidrule(lr){2-4}\cmidrule(lr){5-7}
$N$ & $[-1,1]$ & $\mathcal E_{1.1}$ & $\mathcal E_{1.2}$ & $[-1,1]$ & $\mathcal E_{1.1}$ & $\mathcal E_{1.2}$\\
\midrule
16  & 0.0706 & 0.0776 & 0.1014 & 0.0708 & 0.0727 & 0.0781\\
64  & 0.0684 & 0.0753 & 0.1019 & 0.0706 & 0.0725 & 0.0778\\
128 & 0.0683 & 0.0751 & 0.1015 & 0.0706 & 0.0725 & 0.0778\\
256 & 0.0683 & 0.0751 & 0.1014 & 0.0706 & 0.0725 & 0.0778\\
\bottomrule
\end{tabular}
\end{table}

Table~\ref{tab:ex1} checks the uniformity in $N$ on which Theorem~\ref{thm:main} rests. It lists $N^2\max\norm{U-I}$
on $[-1,1]$ and on the ellipses $\mathcal E_{1.1}$ and $\mathcal E_{1.2}$, where $U$ is continued off the nodes as
$U=(\mathcal I_NU)\bigl((\mathcal I_NU)^\sharp(\mathcal I_NU)\bigr)^{-1/2}$ (Lemma~\ref{lem:polar}
and Theorem~\ref{thm:main}(a)). By the maximum principle, the maximum over an ellipse is attained on its boundary, where we
take $1000$ equispaced points $z$ in~\eqref{eq:ellipse}. The values settle as $N$ grows, on the ellipses as on
$[-1,1]$, consistent with the bound~\eqref{eq:main2} with a constant independent of $N$.

\subsection{Example 2: An exact crossing}\label{sec:ex2}

The second example is the projector onto an eigenvector whose eigenvalue crosses the other one. For a parameter
$a\ne0$, the Hermitian family
\begin{equation}\label{eq:ex2}
	H(k)=kM(k), \qquad M(k)=\begin{pmatrix} a & e^{-\I k}\\ e^{\I k} & -a\end{pmatrix},
\end{equation}
has the eigenvalues $\pm sk$, $s=\sqrt{1+a^2}$, which cross at $k=0$ (Figure~\ref{fig:eigs}(b)). The eigenvectors of
$H(k)$ are those of $M(k)$, whose eigenvalues $\pm s$ never meet. Thus the branch $\lambda(k)=sk$, the lower eigenvalue
for $k<0$ and the upper one for $k>0$, has an eigenvector that is analytic through the crossing, as Rellich's theorem
(Section~\ref{sec:pt}) states. We take $P(k)=\frac12\bigl(I+M(k)/s\bigr)$, the projector onto it. It is entire,
so Theorem~\ref{thm:main} holds with every $\rho>1$. With $\cos\varphi=a/s$, the parallel-transported frame started
from $V_0=W(-1)$ is
\begin{equation}\label{eq:ex2W}
	W(k)=e^{-\I(k+1)\sin^2\frac\varphi2}\begin{pmatrix}\cos\frac\varphi2\\[2pt] e^{\I k}\sin\frac\varphi2\end{pmatrix}.
\end{equation}
We take $a=1$. For $a=0$ the eigenvector runs along a great circle of the Bloch sphere, a geodesic, and the computed
frames are exact up to rounding, as noted after Proposition~\ref{prop:second}; $a=1$ is the generic case.

The eigensolver sorts the eigenvectors by their eigenvalues, so it follows the branch $\lambda=sk$ on one side of the
crossing only. At each node, we therefore take as $Q_{j+1}$ the eigenvector with the larger overlap $\abs{Q^*V_j}$
with the current frame, and we take $N$ odd, so that the crossing $k=0$, where $H(0)=0$, is not a node. Identifying
the eigenvectors of the selected branch near a crossing is a separate problem from transporting them. Here it is solved by the
overlaps: the rejected eigenvector has overlap at most $0.34$ with the current frame, and the chosen one at least
$0.94$.

\begin{figure}[t]
\centering
\includegraphics[width=\textwidth]{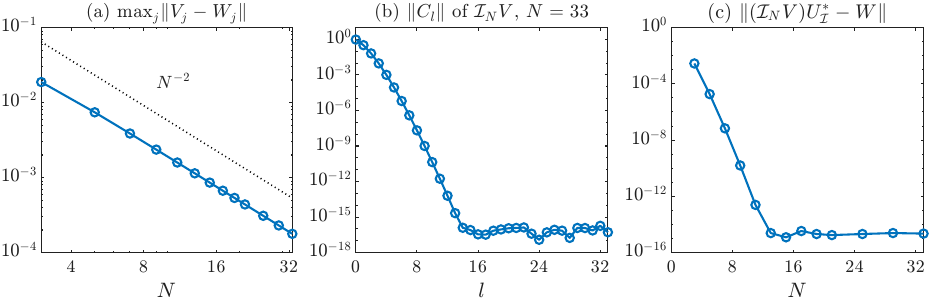}
\caption{Example 2, the exact crossing~\eqref{eq:ex2}, $a=1$. (a) The error of the computed frames at the nodes, with $N^{-2}$ (dotted). (b) The Chebyshev coefficients $C_l$ of
the interpolant $\mathcal I_NV$ of the computed frames for $N=33$. (c) The error of the parallel-transported frame recovered by
Corollary~\ref{cor:recover}, for odd $N$.}
\label{fig:ex2}
\end{figure}

The errors $V_j-W_j$ of the computed frames decrease as $N^{-2}$ (Figure~\ref{fig:ex2}(a)). Since $P$ is entire,
the Chebyshev coefficients of $\mathcal I_NV$ decay faster than any geometric rate (Figure~\ref{fig:ex2}(b)), and the
frame recovered by Corollary~\ref{cor:recover} reaches rounding errors at $N=13$ (Figure~\ref{fig:ex2}(c)).

\subsection{Example 3: A non-normal family}\label{sec:ex3}

The third example is an invariant subspace of a non-normal family. For a parameter $\beta>0$,
the family
\begin{equation}\label{eq:ex3}
	A(k)=\begin{pmatrix}-1 & 2\beta ke^{\I k}\\ 0 & 1\end{pmatrix}
\end{equation}
has the eigenvalues $-1$ and $1$ for every $k$. The eigenvector of $1$ is $x(k)=(\beta ke^{\I k},1)^T$, which is not
orthogonal to the eigenvector $(1,0)^T$ of $-1$. We take $P=xx^*/(x^*x)$ with $x=x(k)$, the orthogonal projector onto it,
so $n=1$. It continues analytically as $xx^\sharp/(x^\sharp x)$, and $x^\sharp x=1+\beta^2k^2$ vanishes at
$\pm\I/\beta$, so $\rho=1/\beta+\sqrt{1+1/\beta^2}$. The rate is thus set by the non-normality alone, not by the
eigenvalues. The parallel-transported frame started from $V_0=W(-1)$ is
\begin{equation}\label{eq:ex3W}
	W(k)=\frac{e^{\I\gamma(k)}}{\sqrt{1+\beta^2k^2}}\,x(k), \qquad
	\gamma(k)=-(k+1)+\frac{\arctan(\beta k)+\arctan\beta}\beta.
\end{equation}
We take $\beta=1$, $2$ and $4$, for which $\rho=2.414$, $1.618$ and $1.281$.

\begin{figure}[t]
\centering
\includegraphics[width=\textwidth]{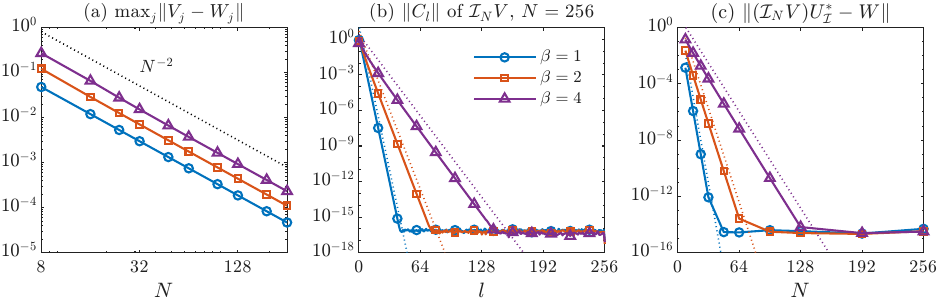}
\caption{Example 3, the non-normal family~\eqref{eq:ex3}, for $\beta=1$, $2$ and $4$. (a) The error of the computed frames at the nodes, with $N^{-2}$ (dotted). (b) The Chebyshev coefficients
$C_l$ of the interpolant $\mathcal I_NV$ of the computed frames for $N=256$, with $\rho^{-l}$ (dotted). (c) The error
of the parallel-transported frame recovered by Corollary~\ref{cor:recover}, with $\rho^{-N}$ (dotted).}
\label{fig:ex3}
\end{figure}

Figure~\ref{fig:ex3} shows the behavior of Example~3, with rates set by $\beta$. The errors $V_j-W_j$ decrease as
$N^{-2}$, with a constant that grows with $\beta$ (Figure~\ref{fig:ex3}(a)). The three families have the same
eigenvalues, but the Chebyshev coefficients of $\mathcal I_NV$ decay at the rate $\rho^{-l}$ (Figure~\ref{fig:ex3}(b)), and
the frame recovered by Corollary~\ref{cor:recover} converges to $W$ at the rate $\rho^{-N}$ (Figure~\ref{fig:ex3}(c)), with
$\rho$ determined by the non-normality.

\subsection{Example 4: Two eigenvalues with a branch point}\label{sec:ex4}

The fourth example is also non-normal, with rank $n=2$, and the two eigenvalues of the subspace meet at a branch
point. For a parameter $\varepsilon>0$, the $4\times4$ family
\begin{equation}\label{eq:ex4}
	A(k)=R(k)T(k)R(k)^*, \qquad R(k)=e^{-\I\varphi(k)S}, \qquad \varphi(k)=\arctan\frac k\varepsilon,
\end{equation}
with
\begin{equation}\label{eq:ex4T}
	T(k)=\begin{pmatrix}T_{00}(k)&0\\0&\operatorname{diag}\bigl(\frac32-\frac k2,\,2+\frac k2\bigr)\end{pmatrix}, \qquad
	T_{00}(k)=\begin{pmatrix}-\frac32&1\\k&-\frac32\end{pmatrix},
\end{equation}
has the eigenvalues $-\frac32\pm\sqrt k$ of $T_{00}(k)$, a complex pair for $k<0$, and $\frac32-\frac k2$ and
$2+\frac k2$, shown in Figure~\ref{fig:eigs}(c). Here $S$ is the fixed Hermitian matrix
\begin{equation}\label{eq:ex4S}
	S=\begin{pmatrix}
	0.3 & 0.5+0.2\I & -0.4+0.1\I & 0.2-0.3\I\\
	0.5-0.2\I & -0.6 & 0.3+0.4\I & -0.1+0.5\I\\
	-0.4-0.1\I & 0.3-0.4\I & 0.1 & 0.6+0.2\I\\
	0.2+0.3\I & -0.1-0.5\I & 0.6-0.2\I & 0.4
	\end{pmatrix}.
\end{equation}
The eigenvalues $-\frac32\pm\sqrt k$ have a branch point at $k=0$, where $T_{00}(0)$ is a Jordan block. Their
eigenvectors $R(k)(1,\pm\sqrt k,0,0)^T$ are not analytic there, and their Riesz projectors have a pole
(Theorem~II.1.9 in~\cite{Kato}), so, unlike the branch $sk$ in Example~2, neither eigenvalue can be treated alone.
Their invariant subspace
is analytic: since $T(k)$ is block diagonal, the orthogonal projector onto it is
$P(k)=R(k)P_0R(k)^*$, $P_0=\operatorname{diag}(1,1,0,0)$. The two eigenvalues stay at a distance of at least
$\frac32$ from the other two. The projector is analytic except at
the branch points $\pm\I\varepsilon$ of $\varphi$, so $\rho=\varepsilon+\sqrt{1+\varepsilon^2}$, as in Example~1. With
$\varphi_0=\varphi(-1)$ and $S_{00}$ the upper left $2\times2$ block of $S$, the parallel-transported frame started
from $V_0=W(-1)$ is
\begin{equation}\label{eq:ex4W}
	W(k)=R(k)\begin{pmatrix}e^{\I(\varphi(k)-\varphi_0)S_{00}}\\ 0\end{pmatrix}.
\end{equation}
We take $\varepsilon=0.5$, so $\rho=1.618$. At each node, $Q_{j+1}$
consists of the first two Schur vectors of an ordered Schur decomposition of $A(k_{j+1})$, with the eigenvalues of
negative real part first. These are orthonormal also at $k=0$, a node for even $N$, where the eigenvectors are
parallel. Unlike in Example~2, no eigenvectors have to be selected, since Algorithm~\ref{alg:dt} depends on the
subspace spanned by $Q_{j+1}$, not on its basis.

\begin{figure}[t]
\centering
\includegraphics[width=\textwidth]{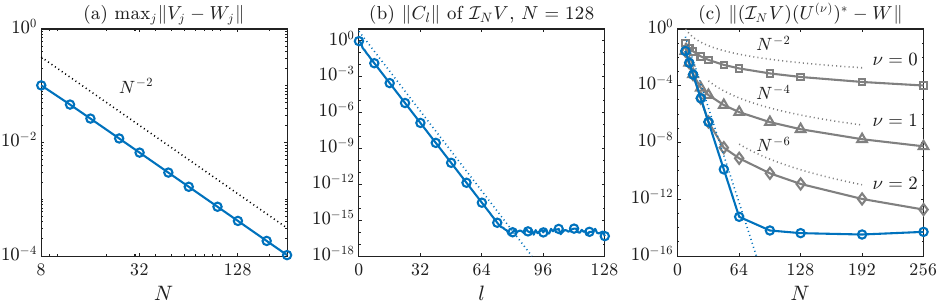}
\caption{Example 4, the family~\eqref{eq:ex4}, $\varepsilon=0.5$. (a) The error of the computed frames at the nodes, with $N^{-2}$ (dotted). (b) The Chebyshev coefficients $C_l$
of the interpolant $\mathcal I_NV$ of the computed frames for $N=128$, with $\rho^{-l}$ (dotted). (c) The error of
$(\mathcal I_NV)(U^{(\nu)})^*$ after $\nu=0,1,2$ Picard iterations~\eqref{eq:picard}, and of the
parallel-transported frame recovered by Corollary~\ref{cor:recover}, with $N^{-2(\nu+1)}$ and $\rho^{-N}$ (dotted).}
\label{fig:ex4}
\end{figure}

The gauges $U_j$ are now $2\times2$ unitary matrices. As in Example~1, the errors $V_j-W_j$ decrease as $N^{-2}$
(Figure~\ref{fig:ex4}(a)), the Chebyshev coefficients of the interpolant $\mathcal I_NV$ of the computed frames decay
at the rate $\rho^{-l}$ down to rounding errors (Figure~\ref{fig:ex4}(b)), and the frame recovered by
Corollary~\ref{cor:recover} converges to $W$ at the rate $\rho^{-N}$ (Figure~\ref{fig:ex4}(c)), through the branch point.
Figure~\ref{fig:ex4}(c) also shows the error of $(\mathcal I_NV)(U^{(\nu)})^*$ after $\nu$ Picard
iterations~\eqref{eq:picard}. The error for $\nu=0$, $\norm{\mathcal I_NV-W}$, decreases only as $N^{-2}$. The errors for $\nu=1$ and $\nu=2$ follow the recovered
frame while $\rho^{-N}$ dominates, and then decrease as $N^{-4}$ and $N^{-6}$, as~\eqref{eq:picardbound} states.

\subsection{Example 5: A large family with a gap}\label{sec:ex5}

The last example is larger, and its parallel-transported frame is not known in closed form. The Hermitian family
\begin{equation}\label{eq:ex5}
	H(k)=H_0+kH_1, \qquad m=1000,
\end{equation}
has $H_0=\operatorname{diag}(a_1,\dots,a_m)$, with $a_1,\dots,a_{20}$ equispaced in $[-2,-1]$ and
$a_{21},\dots,a_m$ equispaced in $[1,2]$, and a fixed random Hermitian matrix $H_1=\frac12(G+G^*)$, where $G$ has
independent standard complex Gaussian entries, scaled so that $\norm{H_1}=0.8$. We take $P(k)$ to be the projector onto
the eigenvectors of the $20$ lowest eigenvalues, so $n=20$. By Weyl's inequality, for $k\in[-1,1]$ these eigenvalues
stay below $-0.2$ and the others above $0.2$, so the gap never closes and $P$ is analytic in a neighborhood of
$[-1,1]$.

Since $W$ is not known, the errors are measured without it. The Chebyshev nodes for $N=240$ contain those for every $N$
that divides $240$, so all runs use the eigenvectors at these $241$ nodes, and the frame $W^{(240)}$ recovered by
Corollary~\ref{cor:recover} at $N=240$ serves as a converged numerical reference. For $N<240$ we
compare the computed frames $V_j$ with $W^{(240)}$ at the nodes, and the recovered frame $W^{(N)}$ with $W^{(240)}$ on
$[-1,1]$. The projector defect $\norm{(\mathcal I_NV)(\mathcal I_NV)^*-P}$ of Theorem~\ref{thm:main}(c) needs only $P$,
which the eigensolver gives at the evaluation points. For the Chebyshev coefficients, the rate $\rho$ is not known.
Instead, we compare them with those of the interpolant $\mathcal I_N(PX)$, for a fixed random $m\times n$ matrix $X$
with orthonormal columns. The function $PX$ is analytic wherever $P$ is, and its values $P(k_j)X=Q_j(Q_j^*X)$ at the
nodes come from the eigensolver alone, without the discrete parallel transport. For a generic $X$, its Chebyshev
coefficients therefore decay at the rate at which $P$ itself can be interpolated.

\begin{figure}[t]
\centering
\includegraphics[width=\textwidth]{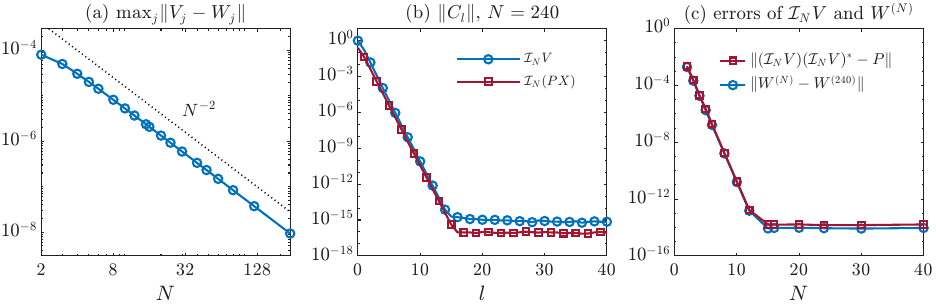}
\caption{Example 5, the family~\eqref{eq:ex5} with $m=1000$ and $n=20$. (a) The error of the computed frames at the
nodes, measured against the recovered frame $W^{(240)}$, with $N^{-2}$ (dotted). (b) The Chebyshev coefficients $C_l$ of
the interpolant $\mathcal I_NV$ of the computed frames for $N=240$, and those of $\mathcal I_N(PX)$ for a fixed random
$X$, which decay at the rate of $P$. (c) The projector defect of $\mathcal I_NV$, and the distance between the recovered
frames $W^{(N)}$ and $W^{(240)}$.}
\label{fig:ex5}
\end{figure}

The errors $V_j-W_j$ of the computed frames decrease as $N^{-2}$ (Figure~\ref{fig:ex5}(a)). The Chebyshev
coefficients of $\mathcal I_NV$ decay at the same rate as those of $\mathcal I_N(PX)$ (Figure~\ref{fig:ex5}(b)), so the
computed frames are interpolated as fast as $P$ itself. The projector defect and the recovered frame reach rounding
errors at $N=15$ (Figure~\ref{fig:ex5}(c)). The recovered frame thus needs only $16$ eigensolves of size $1000$, at the
nodes for $N=15$. At these nodes, the computed frames differ from the parallel-transported frame by $2\times10^{-6}$
(Figure~\ref{fig:ex5}(a)), yet they determine it to rounding errors. Each step of the discrete
parallel transport costs $O(mn^2)$ operations (Section~\ref{sec:compute}), which is negligible next to the eigensolves.

\section{Conclusion}\label{sec:conclusion}

We have shown that discrete parallel transport, despite being a second-order approximation to continuous parallel
transport, produces frames that can be interpolated with spectral accuracy when the underlying family of subspaces is
analytic. The key result is that the computed frames are exact samples of an analytic frame whose gauge relative to the
parallel-transported frame differs from the identity by $O(N^{-2})$ throughout a fixed complex neighborhood. This
uniform analytic structure explains the spectral decay of the interpolation error. Moreover, the gauge can be removed by
solving a small matrix differential equation, recovering the parallel-transported frame to spectral accuracy without
additional eigensolves. Thus, the standard discrete transport step already retains the information needed for spectral
accuracy. The rest of this section discusses generalizations of the results and potential applications.

For families of finite smoothness, numerical experiments indicate that if the $p$-th derivative
of $P$ is of bounded variation, $p\ge3$, the computed frames still differ from the parallel-transported frame by
$O(N^{-2})$, yet the recovered frame converges at the rate $N^{-p}$, the rate of interpolating $W$ itself (Theorem~7.2
in~\cite{ATAP}). For adaptivity, nested grids of Chebyshev points of the second kind reuse the eigensolves under
refinement, and piecewise grids concentrate nodes near avoided crossings, where the region of analyticity is small.

Beyond the interval, families that are periodic in the parameter, or periodic up to a fixed unitary similarity, as for
Wannier functions in one dimension~\cite{GZ25,ZhangTR1575}, are the subject of a sequel to this paper, where equispaced nodes
and trigonometric interpolation replace Chebyshev nodes and polynomial interpolation. In another direction, the results
extend to families of finite-rank orthogonal projectors on a Hilbert space that are analytic in the operator norm. The
estimates of Section~\ref{sec:estimates} hold for operator-valued functions with the same proofs, the frames become
operators from $\C^n$ into the Hilbert space, and the overlap matrices, polar factors and square roots remain
$n\times n$.

The results also extend to a product of intervals by transporting in one direction after
the other, as in the construction of Wannier functions in~\cite{Zhang25}. On the square $[-1,1]^2$, the frame is
transported along the bottom edge and then, from each of its Chebyshev nodes, along the vertical line through it;
further parameters are added in the same way. The resulting frame is analytic in both variables, by the proof
in~\cite{Zhang25} for a single band on the Brillouin zone, a torus. By~\ref{P}, the gauge of the computed frames
factors into the gauge along the bottom edge and those along the vertical lines, and the proofs of
Section~\ref{sec:proofs} hold for the latter with the first variable as a complex parameter, since
Lemma~\ref{lem:polar} allows several variables and the other estimates are uniform in it. A product of intervals is
simply connected, so there is no topological obstruction; on the torus, a nonzero Chern number obstructs a periodic
frame.

The results are most useful when the frame itself, not only the subspace, is needed and the
eigensolves are expensive. Diabatic bases in quantum chemistry are an example: each eigensolve is an
electronic-structure calculation, and the singular-value basis, a fixed reference, is popular in numerical work since
it needs no integration along a path~\cite{LRS22}. The parallel-transported diabatic basis, defined by transport along
radial lines from a reference geometry~\cite{LRS22}, could instead be computed from the electronic states at the
Chebyshev nodes of each line. For Wannier functions, the frame over the Brillouin zone is the output, and its Fourier
coefficients are the Wannier functions~\cite{MV97,Vanderbilt18}; this uses the periodic and several-parameter settings
above. In parametric model reduction, reduced bases at sample parameter values are aligned with a fixed reference
before the reduced operators are interpolated~\cite{AF11}; discrete parallel transport would align them also where
the subspace is far from the reference (Remark~\ref{rmk:fixedref}).

\bibliographystyle{plain}
\bibliography{refs}

\begin{thebibliography}{10}

\bibitem{AF11}
David Amsallem and Charbel Farhat.
\newblock An online method for interpolating linear parametric reduced-order
  models.
\newblock {\em SIAM Journal on Scientific Computing}, 33(5):2169--2198, 2011.

\bibitem{BZA24}
Thomas Bendokat, Ralf Zimmermann, and P.-A. Absil.
\newblock A {G}rassmann manifold handbook: basic geometry and computational
  aspects.
\newblock {\em Advances in Computational Mathematics}, 50(1):6, 2024.

\bibitem{BV96}
J.~Bennetto and David Vanderbilt.
\newblock Semiconductor effective charges from tight-binding theory.
\newblock {\em Physical Review B}, 53(23):15417--15420, 1996.

\bibitem{BDF08}
David Bindel, James~W. Demmel, and Mark Friedman.
\newblock Continuation of invariant subspaces in large bifurcation problems.
\newblock {\em SIAM Journal on Scientific Computing}, 30(2):637--656, 2008.

\bibitem{BG73}
{\AA}ke Bj{\"o}rck and Gene~H. Golub.
\newblock Numerical methods for computing angles between linear subspaces.
\newblock {\em Mathematics of Computation}, 27(123):579--594, 1973.

\bibitem{BBMN91}
Angelika Bunse-Gerstner, Ralph Byers, Volker Mehrmann, and Nancy~K. Nichols.
\newblock Numerical computation of an analytic singular value decomposition of
  a matrix valued function.
\newblock {\em Numerische Mathematik}, 60(1):1--39, 1991.

\bibitem{CLPS17}
{\'E}ric Canc{\`e}s, Antoine Levitt, Gianluca Panati, and Gabriel Stoltz.
\newblock Robust determination of maximally localized {W}annier functions.
\newblock {\em Physical Review B}, 95(7):075114, 2017.

\bibitem{CHN16}
Horia~D. Cornean, Ira Herbst, and Gheorghe Nenciu.
\newblock On the construction of composite {W}annier functions.
\newblock {\em Annales Henri Poincar\'e}, 17(12):3361--3398, 2016.

\bibitem{DDF00}
James~W. Demmel, Luca Dieci, and Mark~J. Friedman.
\newblock Computing connecting orbits via an improved algorithm for continuing
  invariant subspaces.
\newblock {\em SIAM Journal on Scientific Computing}, 22(1):81--94, 2000.

\bibitem{DF01}
Luca Dieci and Mark~J. Friedman.
\newblock Continuation of invariant subspaces.
\newblock {\em Numerical Linear Algebra with Applications}, 8(5):317--327,
  2001.

\bibitem{DPP13}
Luca Dieci, Alessandra Papini, and Alessandro Pugliese.
\newblock Approximating coalescing points for eigenvalues of {H}ermitian
  matrices of three parameters.
\newblock {\em SIAM Journal on Matrix Analysis and Applications},
  34(2):519--541, 2013.

\bibitem{EAS98}
Alan Edelman, Tom{\'a}s~A. Arias, and Steven~T. Smith.
\newblock The geometry of algorithms with orthogonality constraints.
\newblock {\em SIAM Journal on Matrix Analysis and Applications},
  20(2):303--353, 1998.

\bibitem{Frankel11}
Theodore Frankel.
\newblock {\em The Geometry of Physics: An Introduction}.
\newblock Cambridge University Press, Cambridge, third edition, 2011.

\bibitem{GLS19}
David Gontier, Antoine Levitt, and Sami Siraj-dine.
\newblock Numerical construction of {W}annier functions through homotopy.
\newblock {\em Journal of Mathematical Physics}, 60(3):031901, 2019.

\bibitem{GZ25}
Abinand Gopal and Hanwen Zhang.
\newblock A highly accurate procedure for computing globally optimal {W}annier
  functions in one-dimensional crystalline insulators.
\newblock {\em Advances in Computational Mathematics}, 51(6):52, 2025.

\bibitem{HLW06}
Ernst Hairer, Christian Lubich, and Gerhard Wanner.
\newblock {\em Geometric Numerical Integration: Structure-Preserving Algorithms
  for Ordinary Differential Equations}, volume~31 of {\em Springer Series in
  Computational Mathematics}.
\newblock Springer, Berlin, second edition, 2006.

\bibitem{Harland23}
Derek Harland.
\newblock Approximating the parallel transport of an induced connection.
\newblock {\em Proceedings of the Royal Society A}, 479:20220761, 2023.

\bibitem{Higham08}
Nicholas~J. Higham.
\newblock {\em Functions of Matrices: Theory and Computation}.
\newblock SIAM, Philadelphia, 2008.

\bibitem{Kato50}
Tosio Kato.
\newblock On the adiabatic theorem of quantum mechanics.
\newblock {\em Journal of the Physical Society of Japan}, 5(6):435--439, 1950.

\bibitem{Kato}
Tosio Kato.
\newblock {\em Perturbation Theory for Linear Operators}.
\newblock Classics in Mathematics. Springer, Berlin, 1995.

\bibitem{LRS22}
Robert Littlejohn, Jonathan Rawlinson, and Joseph Subotnik.
\newblock The parallel-transported (quasi)-diabatic basis.
\newblock {\em The Journal of Chemical Physics}, 157(18):184303, 2022.

\bibitem{MMYSV12}
Nicola Marzari, Arash~A. Mostofi, Jonathan~R. Yates, Ivo Souza, and David
  Vanderbilt.
\newblock Maximally localized {W}annier functions: theory and applications.
\newblock {\em Reviews of Modern Physics}, 84(4):1419--1475, 2012.

\bibitem{MV97}
Nicola Marzari and David Vanderbilt.
\newblock Maximally localized generalized {W}annier functions for composite
  energy bands.
\newblock {\em Physical Review B}, 56(20):12847--12865, 1997.

\bibitem{MR93}
Volker Mehrmann and Werner Rath.
\newblock Numerical methods for the computation of analytic singular value
  decompositions.
\newblock {\em Electronic Transactions on Numerical Analysis}, 1:72--88, 1993.

\bibitem{Rheinboldt88}
Werner~C. Rheinboldt.
\newblock On the computation of multi-dimensional solution manifolds of
  parametrized equations.
\newblock {\em Numerische Mathematik}, 53:165--181, 1988.

\bibitem{Roman06}
Eric Roman, Jonathan~R. Yates, Marek Veithen, David Vanderbilt, and Ivo Souza.
\newblock Ab initio study of the nonlinear optics of {III}-{V} semiconductors
  in the terahertz regime.
\newblock {\em Physical Review B}, 74(24):245204, 2006.

\bibitem{SteinShakarchi03}
Elias~M. Stein and Rami Shakarchi.
\newblock {\em Complex Analysis}.
\newblock Princeton Lectures in Analysis II. Princeton University Press,
  Princeton, 2003.

\bibitem{Teschl12}
Gerald Teschl.
\newblock {\em Ordinary Differential Equations and Dynamical Systems}, volume
  140 of {\em Graduate Studies in Mathematics}.
\newblock American Mathematical Society, Providence, RI, 2012.

\bibitem{ATAP}
Lloyd~N. Trefethen.
\newblock {\em Approximation Theory and Approximation Practice}.
\newblock SIAM, Philadelphia, extended edition, 2019.

\bibitem{Vanderbilt18}
David Vanderbilt.
\newblock {\em Berry Phases in Electronic Structure Theory: Electric
  Polarization, Orbital Magnetization and Topological Insulators}.
\newblock Cambridge University Press, Cambridge, 2018.

\bibitem{WST16}
Georg~W. Winkler, Alexey~A. Soluyanov, and Matthias Troyer.
\newblock Smooth gauge and {W}annier functions for topological band structures
  in arbitrary dimensions.
\newblock {\em Physical Review B}, 93(3):035453, 2016.

\bibitem{ZhangTR1575}
Hanwen Zhang.
\newblock Constructing globally optimal {W}annier functions for isolated
  composite bands in one dimension: analysis.
\newblock Research Report YALEU/DCS/TR-1575, Yale University, Department of
  Computer Science, 2025.

\bibitem{Zhang25}
Hanwen Zhang.
\newblock Constructing optimal {W}annier functions via potential theory:
  isolated single band for matrix models.
\newblock arXiv:2502.08641, 2025.

\end{thebibliography}

\end{document}